\documentclass[11pt]{article}

\usepackage[utf8]{inputenc}
\usepackage[english]{babel}
\usepackage{amsmath}
\usepackage{nicefrac}
\usepackage{amsthm}
\usepackage{amsfonts}
\usepackage{verbatim}
\usepackage{color}
\usepackage[arrow, matrix, curve]{xy}
\usepackage{bbm}
\usepackage{eqparbox}
\usepackage[numbers]{natbib}
\usepackage{stmaryrd}
\usepackage{amssymb}
\usepackage{mathrsfs}
\usepackage{pictexwd,dcpic}
\usepackage{mathabx}

\DeclareMathSymbol{\leq}{\mathrel}{symbols}{20}
   
\DeclareMathSymbol{\geq}{\mathrel}{symbols}{21}
   \let\ge=\geq

\newtheoremstyle{WreschTheoremstyle} 
                        {1.5em}    
                        {2.5em}    
                        {}         
                        {}         
                        {\bfseries}
                        {}        
                        {\newline} 
                        {\raisebox{0.6em}{\thmname{#1}\thmnumber{#2}\thmnote{ (#3)}}}

\newcommand{\R}{\mathbb{R}}

\newcommand{\N}{\mathbb{N}}

\newcommand{\F}{\mathcal{F}}

\renewcommand{\P}{\mathbb{P}}
\newcommand{\E}{\mathbb{E}}
\newcommand{\e}{\varepsilon}

\renewcommand{\1}{\mathbbm{1}}

\newtheorem{Theorem}{Theorem}[section]
\newtheorem{Proposition}[Theorem]{Proposition}
\newtheorem{Corollary}[Theorem]{Corollary}
\newtheorem{Lemma}[Theorem]{Lemma}
\newtheorem{Remark}[Theorem]{Remark}

\newtheorem{Example}[Theorem]{Example}

\numberwithin{equation}{section}

\makeatletter
\newcommand{\customlabel}[1]{%
     \stepcounter{ref}%
   \protected@write
\@auxout{}{\string\newlabel{#1}{{\thesatz.\arabic{ref}}{\thepage}{\thesatz.\arabic{ref}}{#1}{}}}%
   \hypertarget{#1}{\thesatz.\arabic{ref}}%
}
\makeatother

\newenvironment{sciabstract}{\begin{quote}}{\end{quote}}

\newcounter{lastnote}

\title{Existence of densities for stochastic differential equations driven by L\'{e}vy processes with anisotropic jumps}

\author{
 Martin Friesen\footnote{Fakult\"at f\"ur Mathematik und Naturwissenschaften, Bergische Universit\"at Wuppertal, Gaußstraße 20, 42119 Wuppertal, Germany, friesen@math.uni-wuppertal.de}\\
 Peng Jin\footnote{Department of Mathematics, Shantou University, Shantou, Guangdong 515063, China, pjin@stu.edu.cn}\\
 Barbara R\"udiger\footnote{Fakult\"at f\"ur Mathematik und Naturwissenschaften, Bergische Universit\"at Wuppertal, Gaußstraße 20, 42119 Wuppertal, Germany, ruediger@uni-wuppertal.de}
}

\def\HyPsd@CatcodeWarning#1{}

\makeatother

\begin{document}



\maketitle

\begin{sciabstract}\textbf{Abstract:} We study existence of densities
for solutions to stochastic differential equations with Hölder continuous
coefficients and driven by a $d$-dimensional Lévy process $Z=(Z_{t})_{t\geq 0}$,
where, for $t>0$, the density function $f_{t}$ of $Z_{t}$ exists and satisfies, for some $(\alpha_{i})_{i=1,\dots,d}\subset(0,2)$
and $C>0$,
\[
\limsup\limits _{t\to0}t^{1/\alpha_{i}}\int\limits _{\R^{d}}|f_{t}(z+e_{i}h)-f_{t}(z)|dz\leq C|h|,\ \ h\in\R,\ \ i=1,\dots,d.
\]
Here $e_{1},\dots,e_{d}$ denote the canonical basis vectors in $\R^{d}$.
The latter condition covers anisotropic $(\alpha_{1},\dots,\alpha_{d})$-stable
laws but also particular cases of subordinate Brownian motion. 
To prove our result we use some ideas taken from \citep{DF13}. \end{sciabstract}


\section{Introduction}

Let $d\geq1$ and $X=(X(t))_{t\ge0}$ be a solution to the Lévy driven SDE
\begin{align}
dX(t)=b(X(t))dt+\sigma(X(t-))dZ(t)\label{EQ:12}
\end{align}
on $\R^{d}$, where $b\in\R^{d}$ is the drift, $\sigma\in\R^{d\times d}$
the diffusion coefficient and $Z$ a $d$-dimensional
pure jump Lévy process with Lévy measure $\nu$, i.e.,
\begin{align}
\E[e^{i\xi\cdot Z(t)}]=e^{-t\Psi_{\nu}(\xi)},\qquad\Psi_{\nu}(\xi)=\int\limits _{\R^{d}}\left(1+\1_{\{|z|\leq1\}}(\xi\cdot z)i-e^{i(\xi\cdot z)}\right)\nu(dz).\label{SYMBOL}
\end{align}
If the coefficients of (\ref{EQ:12}) are smooth
enough and the driving noise exhibits certain regularities,
then existence of a smooth density for $X_{t}$ can
be obtained by the Malliavin calculus (see, e.g., \citep{BC86}, \citep{P96},
\citep{Z16} and \citep{Z17}). For Hölder continuous coefficients,
as studied in this work, one possibility is to apply the parametrix
method, see, e.g., \citep{KL16}, \citep{CZ16}, \citep{KK18} and
\citep{K18}. In this case one assumes that $b,\sigma$ are bounded,
$\sigma$ is uniformly elliptic, i.e., $\inf_{x\in\R^{d}}\inf_{|v|=1}|\sigma(x)v|>0$,
and most of the results are designed for a driving noise $Z$ which
is comparable to a
\begin{enumerate}
\item[(Z1)] rotationally symmetric $\alpha\in(0,2)$-stable Lévy process with symbol $\Psi_{\nu}(\xi)=|\xi|^{\alpha}$, $\xi\in\R^{d}$.
\end{enumerate}
Recently, in \citep{DF13} a simple method for proving existence of
a density on
\begin{align}
\Gamma=\{x\in\R^{d}\ |\ \sigma(x)\text{ is invertible}\}\label{MAIN:02}
\end{align}
for solutions to \eqref{EQ:12} with bounded and Hölder continuous
coefficients has been developed. Their main condition was formulated
in terms of an integrability and non-degeneracy condition on the Lévy
measure $\nu$ from which, in particular, the following
crucial estimate was derived
\begin{align}
c|\xi|^{\alpha}\leq\mathrm{Re}(\Psi_{\nu}(\xi))\leq C|\xi|^{\alpha},\ \ \forall\xi\in\R^{d},\ \ |\xi|\gg1,\label{EQ:31}
\end{align}
where $c,C>0$ and $\alpha\in(0,2)$ (see also \citep{KS17} for a
discussion of \eqref{EQ:31}). Similar ideas have been first applied
in dimension $d=1$, see \citep{FP10}, where the density was studied
in terms of its characteristic function. The technique in \citep{DF13}
has been further successfully applied in \citep{DR14}, \citep{R16b}
and \citep{R16a} to the Navier-Stokes equations driven by Gaussian
noises in $d=3$, and in \citep{F15} to the space-homogeneous Boltzmann
equation. A summary of this method and further extensions can be found
in \citep{R17}.

In the recent years we observe an increasing interest in the study
of \eqref{EQ:12} for a Lévy process $Z$ with anisotropic jumps (see,
e.g., \citep{BC06}, \citep{BSK17}, \citep{KR17} and \citep{C18}).
The most prominent example here is
\begin{enumerate}
\item[(Z2)] $Z=(Z_{1}^{\alpha_{1}},\dots,Z_{d}^{\alpha_{d}})$, where $Z_{1}^{\alpha_{1}},\dots,Z_{d}^{\alpha_{d}}$
are independent, one-dimensional symmetric $\alpha_{1},\dots,\alpha_{d}\in(0,2)$-stable
Lévy processes, i.e., one has for some constants $c_{\alpha_{1}},\dots,c_{\alpha_{d}}>0$
\begin{align}
\nu(dz)=\sum\limits _{j=1}^{d}c_{\alpha_{j}}\frac{dz_{j}}{|z_{j}|^{1+\alpha_{j}}}\otimes\prod\limits _{k\neq j}\delta_{0}(dz_{k}),\qquad\Psi_{\nu}(\xi)=|\xi_{1}|^{\alpha_{1}}+\dots+|\xi_{d}|^{\alpha_{d}}.\label{EQ:14}
\end{align}
\end{enumerate}
Since (Z2) satisfies \eqref{EQ:31} only for the case where $\alpha_{1}=\dots=\alpha_{d}$,
one obtains existence of a density from \citep{DF13}, provided $b,\sigma$
are bounded, Hölder continuous and $\alpha_{1}=\dots=\alpha_{d}$.
This work provides the first result on the existence of a density
for solutions to \eqref{EQ:12} where $Z$ is given
by (Z2) and $\alpha_{1},\dots,\alpha_{d}$ may be different.

\section{Statement of the results}

\subsection{Assumption on the Lévy noise}

Consider \eqref{EQ:12} for a Lévy process $Z=(Z(t))_{t\ge0}$
with Lévy measure $\nu$ and symbol \eqref{SYMBOL}. We suppose that
$Z$ satisfies the following condition:
\begin{enumerate}
\item[(A1)] For each $t>0$, $Z(t)$ has density $f_{t}$ and
there exist $\alpha_{1},\dots,\alpha_{d}\in(0,2)$ and a constant
$C>0$ such that, for any $k\in\{1,\dots,d\}$,
\begin{align}
\limsup\limits _{t\to0}t^{1/\alpha_{k}}\int\limits _{\R^{d}}|f_{t}(z+e_{k}h)-f_{t}(z)|dz\leq C|h|,\qquad h\in\R.\label{EXAMPLE:10}
\end{align}
\end{enumerate}
The anisotropic regularization property of the noise is reflected
by \eqref{EXAMPLE:10}. It imposes a growth condition for the Besov
norm of $f_{t}$ at the singularity in $t=0$ (since $f_{t}|_{t=0}=\delta_{0}$).
\begin{Remark}
\begin{enumerate}
\item[(i)] If $Z$ has, for $t>0$, density $f_{t}\in C^{1}(\R^{d})$, then
\eqref{EXAMPLE:10} is implied by
\begin{align}
\limsup\limits _{t\to0}t^{1/\alpha_{k}}\left\Vert \frac{\partial f_{t}}{\partial z_{k}}\right\Vert _{L^{1}(\R^{d})}<\infty,\ \ k\in\{1,\dots,d\}.\label{EXAMPLE:30}
\end{align}
\item[(ii)] Suppose that the symbol $\Psi_{\nu}$ given by \eqref{SYMBOL} satisfies
\eqref{EQ:31} for some $\alpha\in(0,2)$. Then $Z$ has, for $t>0$,
a smooth density which satisfies \eqref{EXAMPLE:30} for $\alpha=\alpha_{1}=\dots=\alpha_{d}$,
see, e.g., \citep[Lemma 3.3]{DF13}.
\item[(iii)] Let $W$, $Z$ be independent Lévy processes. If $Z$ satisfies (A1)
for some $(\alpha_{k})_{k\in\{1,\dots,d\}}$, then also $Z+W$ satisfies
(A1) for the same $(\alpha_{k})_{k\in\{1,\dots,d\}}$.
\end{enumerate}
\end{Remark} The following is our main guiding example. \begin{Example}\label{MAIN:EXAMPLE}
Take $m\in\N$ with $m\leq d$, $I_{1},\dots,I_{m}\subset\{1,\dots,d\}$
disjoint with $I_{1}\cup\dots\cup I_{m}=\{1,\dots,d\}$ and $\alpha_{1},\cdots\alpha_{m}\in(0,2)$.
Let $Z$ be a Lévy process with symbol
\[
\Psi_{\nu}(\xi)=\sum\limits _{j=1}^{m}|\xi_{I_{j}}|^{\alpha_{j}},\ \ \xi_{I_{j}}=(\xi_{k})_{k\in I_{j}}\in\R^{I_{j}},\ \ \xi=(\xi_{I_{1}},\dots,\xi_{I_{m}})\in\R^{d}.
\]
Then $Z=(Z_{I_{1}},\dots,Z_{I_{m}})$, where the components $Z_{I_{1}},\dots,Z_{I_{m}}$
are independent and each $Z_{I_{j}},$ $j\in\{1,\ldots m\}$,
is a rotationally symmetric $\alpha_{j}$-stable process in $\R^{I_{j}}$.
Denote by $f_{t}^{I_{1}},\dots,f_{t}^{I_{m}}$ their densities. Then
$Z(t)$ has smooth density $f_{t}(z)=f_{t}^{I_{1}}(z_{I_{1}})\cdots f_{t}^{I_{m}}(z_{I_{m}})$
and hence satisfies \eqref{EXAMPLE:30}.
\begin{enumerate}
\item[(i)] The fully isotropic case (Z1) corresponds to $m=1$, $\alpha:=\alpha_{1}$
and $I_{1}=\{1,\dots,d\}$.
\item[(ii)] The fully anisotropic case (Z2) corresponds to $m=d$ and $I_{j}=\{j\}$.
\end{enumerate}
\end{Example} It is worthwile to mention that (A1) includes a wider
class of Lévy processes than those considered in Example \ref{MAIN:EXAMPLE},
see Section 7.

\subsection{Anisotropic Besov space}

Anisotropic smoothness related to processes given as in (A1) will
be measured by an anisotropic analogue of classical Besov spaces as
introduced below. For $\alpha_{1},\dots,\alpha_{d}$ given as in condition
(A1), we define an anisotropy $(a_{1},\dots,a_{d})$ and mean order
of smoothness $\overline{\alpha}>0$ by
\begin{align}
\frac{1}{\overline{\alpha}}=\frac{1}{d}\left(\frac{1}{\alpha_{1}}+\dots+\frac{1}{\alpha_{d}}\right),\qquad a_{i}=\frac{\overline{\alpha}}{\alpha_{i}},\ \ i=1,\dots,d.\label{MAIN:04}
\end{align}
More generally, we call a collection of numbers $a=(a_{1},\dots,a_{d})$
an anisotropy if it satisfies
\begin{align}
0<a_{1},\dots,a_{d}<\infty\qquad\text{ and }\qquad a_{1}+\dots+a_{d}=d.\label{EQ:35}
\end{align}
Note that $(a_{1},\dots,a_{d})$ given by \eqref{MAIN:04} is an anisotropy
in the sense of \eqref{EQ:35}.

Let $a=(a_{1},\dots,a_{d})$ be an anisotropy and take $\lambda>0$
with $\lambda/a_{k}\in(0,1)$ for all $k\in\{1,\dots,d\}$. The anisotropic
Besov space $B_{1,\infty}^{\lambda,a}(\R^{d})$ is defined as the
Banach space of functions $f:\R^{d}\longrightarrow\R$ with finite
norm
\begin{align}
\|f\|_{B_{1,\infty}^{\lambda,a}}:=\|f\|_{L^{1}(\R^{d})}+\sum\limits _{k=1}^{d}\sup\limits _{h\in[-1,1]}|h|^{-\lambda/a_{k}}\|\Delta_{he_{k}}f\|_{L^{1}(\R^{d})},\label{EQ:36}
\end{align}
where $\Delta_{h}f(x)=f(x+h)-f(x)$, $h\in\R^{d}$, and $e_{k}\in\R^{d}$
denotes the canonical basis vector in the $k$-the direction (see
\citep{D03} and \citep{T06} for additional details and references).
In the above definition, $\lambda/a_{k}$ describes the smoothness
in the $k$-th coordinate, its restriction to $(0,1)$
is not essential. Without this restriction we should use iterated
differences in \eqref{EQ:36} instead (see \citep[Theorem 5.8.(ii)]{T06}).

\subsection{Some notation}

For a $d\times d$ matrix $A$, we set $|A|=\sup_{|x|=1}|Ax|$. Then
$1/|A^{-1}|=\inf_{|x|=1}|Ax|$ with the convention that $1/|A^{-1}|:=0$
if $A$ is not invertible. Given another $d\times d$ matrix $B$,
we deduce that
\begin{align}
\frac{1}{|A^{-1}|}\leq|A|,\qquad\left|\frac{1}{|A^{-1}|}-\frac{1}{|B^{-1}|}\right|\leq|A-B|.\label{MATRIX:INEQ}
\end{align}
Here and below we denote by $C$ a generic positive constant which
may vary from line to line. Possible dependencies on other parameters
are denoted by $C=C(a,b,c,\dots)$. For $a,b\geq0$, we set $a\wedge b:=\min\{a,b\}$
and denote by $\lfloor a\rfloor$ the unique integer satisfying $a\leq\lfloor a\rfloor<a+1$.

Here and below we say that $X$ is a solution to \eqref{EQ:12}, if it is a weak solution to \eqref{EQ:12} in the following sense:
There exist a stochastic basis $(\Omega, \F, (\F_t)_{t \geq 0}, \P)$ with the usual conditions,
an $(\F_t)_{t \geq 0}$-L\'evy process $Z$,
and an $(\F_t)_{t \geq 0}$-adapted c\'adl\'ag process $X$ such that \eqref{EQ:12} is satisfied.
In order to simplify the notation, we simply write $X$ for the (weak) solution to \eqref{EQ:12}.
Note that under the conditions imposed in this work existence and uniqueness of solutions to \eqref{EQ:12}
not fully established.
The classical existence and uniqueness theory for stochastic equations with jumps is discussed, e.g., in
\citep{JS03,S05}, see also \citep{CZZ17} and \citep{XZ17} for some
recent results in this direction.

\subsection{General case}

We study \eqref{EQ:12} for bounded and Hölder continuous coefficients,
i.e., we suppose that:
\begin{enumerate}
\item[(A2)] $b,\sigma$ are bounded, and there exist $\beta\in[0,1]$, $\chi\in(0,1)$
and $C>0$ such that
\[
|b(x)-b(y)|\leq C|x-y|^{\beta},\qquad|\sigma(x)-\sigma(y)|\leq C|x-y|^{\chi},\ \ x,y\in\R^{d}.
\]
\end{enumerate}
Note that $\beta=0$ corresponds to the case where $b$ is only bounded.
Let $\alpha_{1},\ldots,\alpha_{d}$ be as in condition (A1). Set $\alpha^{\mathrm{min}}=\min\{\alpha_{1},\dots,\alpha_{d}\}$
and $\alpha^{\mathrm{max}}=\max\{\alpha_{1},\dots,\alpha_{d}\}$.
The following is our main result. \begin{Theorem}\label{THEOREM:01}
Let $Z$ be a Lévy process with Lévy measure $\nu$ and symbol \eqref{SYMBOL}.
Assume that $Z$ satisfies condition (A1) and there exist constants
$\gamma\in(0,2]$ and $\delta\in(0,\gamma]$ such that
\begin{align}
\int\limits _{\R^{d}}\left(\1_{\{|z|\leq1\}}|z|^{\gamma}+\1_{\{|z|>1\}}|z|^{\delta}\right)\nu(dz)<\infty.\label{MOMENT:00}
\end{align}
Finally, suppose that condition (A2) is satisfied, and the constants
$\gamma,\delta,\beta,\chi$ satisfy:
\begin{enumerate}
\item[(a)] If $\gamma\in[1,2]$, then suppose that
\begin{align}
\alpha^{\mathrm{min}}\left(1+\frac{\beta\wedge\delta}{\gamma}\right)>1,\qquad\frac{\alpha^{\mathrm{min}}}{\gamma}\left(1+\chi\wedge(\delta/\gamma)\right)>1.\label{EQ:01}
\end{align}
\item[(b)] If $\gamma\in(0,1)$, then suppose that
\begin{align}
\alpha^{\mathrm{min}}\left(1+\frac{\beta\wedge\chi}{\gamma}\right)>1,\qquad\alpha^{\mathrm{min}}\left(\frac{1}{\gamma}+\chi\right)>1,\qquad\alpha^{\mathrm{min}}+\beta\wedge\chi>1.\label{EQ:02}
\end{align}
\end{enumerate}
Let $(X(t))_{t\geq0}$ satisfy \eqref{EQ:12} and denote by $\mu_{t}$
the law of $X(t)$. Then there exists $\lambda\in(0,1)$ and
$C>0$ such that the measure $|\sigma^{-1}(x)|^{-1}\mu_{t}(dx)=g(x)dx$ with
\begin{align}
\|g_{t}\|_{B_{1,\infty}^{\lambda,a}}\leq\frac{C}{(1\wedge t)^{1/\alpha^{\mathrm{min}}}},\ \ \ t>0.\label{MAIN:03}
\end{align}
In particular, $\mu_{t}$ has a density on $\Gamma$
(defined in \eqref{MAIN:02}). 
\end{Theorem}
The proof of this statement
is given in Section 5. Without assuming finite moments for the big
jumps as in \eqref{MOMENT:00}, we still get existence of a density
but may lose Besov regularity. 
\begin{Corollary}\label{THEOREM:02}
Let $Z$ be a Lévy process with Lévy measure $\nu$ and symbol \eqref{SYMBOL}.
Assume that $Z$ satisfies condition (A1) and there exists $\gamma\in(0,2]$
such that
\[
\int\limits _{|z|\leq1}|z|^{\gamma}\nu(dz)<\infty.
\]
Moreover, assume that condition (A2) is satisfied. Finally suppose
that either condition \eqref{EQ:01} holds for $\delta=\gamma$ or
\eqref{EQ:02} is satisfied. Let $(X(t))_{t\geq0}$ satisfy
\eqref{EQ:12} and denote by $\mu_{t}$ its distribution.
Then $\mu_{t}$ has a density when restricted to $\Gamma$. 
\end{Corollary} 
\begin{proof} Repeat the arguments given in \citep[Corollary 1.3]{DF13}. 
\end{proof} Let us briefly
comment on the appearance of different conditions for $\gamma\in(0,1)$
and $\gamma\in[1,2]$. If $\gamma\in(0,1)$, then $Z$ has the representation
$Z(t)=Y(t)-t\int_{|z|\leq1}z\nu(dz)$, where $(Y(t))_{t\geq0}$ is
a Lévy process with symbol
\[
\Psi_{Y}(\xi)=\int\limits _{\R^{d}}\left(1-e^{i(\xi\cdot z)}\right)\nu(dz),\ \ \xi\in\R^{d}.
\]
Note that in this case $(Y(t))_{t\geq0}$ satisfies, by Lemma \ref{LEMMA:00}.(b), for any $\eta\leq\delta\leq\gamma<1$,
\[
\E\left[\left(\sum\limits _{u\in[s,t]}|\Delta Y(u)|\right)^{\eta}\right]\leq C(t-s)^{\eta/\gamma},\ \ 0\leq s\leq t\leq s+1.
\]
Hence, the drift $t\int_{|z|\leq1}z\nu(dz)$ is the
dominating part of the Lévy process $Z$, which requires that the
approximation used in the proof of this work has to be adapted appropriately.
For this purpose we employ some ideas taken from \citep{DF13}. A
similar problem was also encountered in \citep{P15}, \citep{K18}
and \citep{KK18}, where similar conditions to \eqref{EQ:02} have
been used.

We close this section with the following remarks. \begin{Remark}
\begin{enumerate}
\item[(a)] Using additional moment estimates, it is not difficult to extend
all results obtained in this work also to the case where $b,\sigma$
have at most linear growth, see \cite{FJR18} for such arguments.
\item[(b)] It is also possible to study the case where the noise is absent in
certain directions, i.e., equations of the form
\begin{align*}
dX(t) & =b(X(t),Y(t))dt,\\
dY(t) & =c(X(t),Y(t))dt+\sigma(X(t),Y(t))dZ(t).
\end{align*}
Under suitable assumptions, the technique used in
this paper can be adapted to prove existence of a density for $Y(t)$,
$t>0$. Such type of equations have also been studied
in \citep{Z14,HM16,Z16,Z17}.
\end{enumerate}
\end{Remark}

\subsection{Diagonal case}

If $\sigma$ in \eqref{EQ:12} is diagonal, one may expect that previous conditions are too rough. 
Below we show that this is indeed the case, i.e., we study the system of stochastic equations
\begin{align}
dX_{k}(t)=b_{k}(X(t))dt+\sigma_{k}(X(t-))dZ_{k}(t),\ \ t\geq0,\ \ k\in\{1,\dots,d\},\label{EXAMPLE:09}
\end{align}
where $\sigma_{1},\dots,\sigma_{d}:\R^{d}\longrightarrow\R$ and the
drift $b$ satisfy:
\begin{enumerate}
\item[(A3)] $b=(b_{1},\dots,b_{d})$ and $\sigma= \mathrm{diag}(\sigma_{1},\dots,\sigma_{d})$
are bounded, and there exist $\beta_{1},\dots,\beta_{d}\in[0,1]$,
$\chi_{1},\dots,\chi_{d}\in(0,1)$ and $C>0$ such that
\[
|b_{k}(x)-b_{k}(y)|\leq C|x-y|^{\beta_{k}},\qquad|\sigma_{k}(x)-\sigma_{k}(y)|\leq C|x-y|^{\chi_{k}}.
\]
\end{enumerate}
Recall that the anisotropy $(a_{1},\dots,a_{d})$
and mean order of smoothness $\overline{\alpha}$ have been defined
in \eqref{MAIN:04}. The following is our main result in this case.
\begin{Theorem}\label{THEOREM:03} Let $Z$ be a Lévy process with
Lévy measure $\nu$ and symbol as in \eqref{SYMBOL}. Assume that
$Z$ satisfies condition (A1) and, for each $k\in\{1,\dots,d\}$,
there exist $\gamma_{k}\in(0,2]$ and $\delta_{k}\in(0,\gamma_{k}]$
such that
\begin{align}
\int\limits _{\R^{d}}\left(\1_{\{|z|\leq1\}}|z_{k}|^{\gamma_{k}}+\1_{\{|z|>1\}}|z_{k}|^{\delta_{k}}\right)\nu(dz)<\infty,\qquad k\in\{1,\dots,d\}.\label{EQ:05}
\end{align}
Set $\gamma=\max\{\gamma_{1},\dots,\gamma_{d}\}$ and $\delta=\min\{\delta_{1},\dots,\delta_{d}\}$.
Assume that condition (A3) is satisfied and, for each $k\in\{1,\dots,d\}$,
\begin{enumerate}
\item[(a)] if $\gamma_{k}\in[1,2]$, suppose that
\begin{align}
\alpha_{k}\left(1+\frac{\beta_{k}\wedge\delta}{\gamma}\right)>1,\qquad\alpha_{k}\left(\frac{1}{\gamma_{k}}+\frac{\chi_{k}\wedge(\delta/\gamma_{k})}{\gamma}\right)>1;\label{EQ:03}
\end{align}
\item[(b)] if $\gamma_{k}\in(0,1)$, suppose that $\alpha^{\mathrm{min}}+\min\{\beta_{1}\wedge\chi_{1},\dots,\beta_{d}\wedge\chi_{d}\}>1$
and
\begin{align}
\alpha_{k}\left(1+\frac{\beta_{k}\wedge\chi_{k}}{\gamma}\right)>1,\ \ \alpha_{k}\left(\frac{1}{\gamma_{k}}+\frac{\chi_{k}}{\max\{1,\gamma\}}\right)>1.\label{EQ:06}
\end{align}
\end{enumerate}
Let $(X(t))_{t\geq0}$ satisfy \eqref{EXAMPLE:09} and denote by $\mu_{t}$
the law of $X_{t}$. Then there exists $\lambda\in(0,1)$ such that
$|\sigma^{-1}(x)|^{-1}\mu_{t}(dx)=g(x)dx$ with $g\in B_{1,\infty}^{\lambda,a}(\R^{d})$
satisfying \eqref{MAIN:03}. In particular, $\mu_{t}$ has a density
on
\[
\Gamma=\{x\in\R^{d}\ |\ \sigma_{k}(x)\neq0,\ \ k\in\{1,\dots,d\}\}.
\]
\end{Theorem} 
A proof of this Theorem is given in Section 6.
Note that in condition (A3) one may always replace $\beta_1,\dots,\beta_d$ by $\min\{\beta_1,\dots, \beta_d\}$ and 
$\chi_1,\dots, \chi_d$ by $\min\{\chi_1,\dots, \chi_d\}$. However, conditions \eqref{EQ:03} and \eqref{EQ:06} would 
become more restrictive in this case. Hence, it is not artificial to assume different H\"older regularity for different components.
We have the following analogue of Corollary \ref{THEOREM:02}. 
\begin{Corollary}\label{THEOREM:04}
Let $Z$ be a Lévy process with Lévy measure $\nu$ and symbol as
in \eqref{SYMBOL}. Assume that $Z$ satisfies condition (A1) and,
for each $k\in\{1,\dots,d\}$, there exist $\gamma_{k}\in(0,2]$ such
that
\begin{align}
\int\limits _{|z|\leq1}|z_{k}|^{\gamma_{k}}\nu(dz)<\infty.\label{APPROX:08}
\end{align}
Assume that condition (A3) holds and, for $k\in\{1,\dots,d\}$,
\begin{enumerate}
\item[(a)] if $\gamma_{k}\in[1,2]$, then suppose that, for $\gamma_{*}=\min\{\gamma_{1},\dots,\gamma_{d}\}$,
\begin{align*}
\alpha_{k}\left(1+\frac{\beta_{k}}{\gamma}\right)>1,\qquad\alpha_{k}\left(\frac{1}{\gamma_{k}}+\frac{\chi_{k}\wedge(\gamma_{*}/\gamma_{k})}{\gamma}\right)>1;
\end{align*}
\item[(b)] if $\gamma_{k}\in(0,1)$, then suppose that $\alpha^{\mathrm{min}}+\min\{\beta_{1}\wedge\chi_{1},\dots,\beta_{d}\wedge\chi_{d}\}>1$
and
\[
\alpha_{k}\left(1+\frac{\beta_{k}\wedge\chi_{k}}{\gamma}\right)>1,\ \ \alpha_{k}\left(\frac{1}{\gamma_{k}}+\frac{\chi_{k}}{\max\{1,\gamma\}}\right)>1.
\]
\end{enumerate}
Let $(X(t))_{t\geq0}$ satisfy \eqref{EXAMPLE:09} and denote by $\mu_{t}$
the law of $X_{t}$. Then $\mu_{t}$ has a density on
\[
\Gamma=\{x\in\R^{d}\ |\ \sigma_{k}(x)\neq0,\ \ k\in\{1,\dots,d\}\}.
\]
\end{Corollary} Let us close the presentation of our results with
one additional remark. 
\begin{Remark} 
Using similar ideas to the
proof of Theorem \ref{THEOREM:03}, it is possible
to obtain sharper results for the case where $\sigma$ has a block
structure so that $\sigma= \mathrm{diag}(\sigma_{1},\ldots,\sigma_{m})$
with $\sigma_{j}:\R^{I_{j}}\longrightarrow\R^{I_{j}}$, $j\in\{1,\dots,m\}$,
and $Z$ is given as in Example \ref{MAIN:EXAMPLE}. 
\end{Remark}

\subsection{Structure of the work}

This work is organized as follows. In Section 3 we discuss the particular
case where $Z$ is either given by (Z1) or (Z2) and also
the case where the drift vanishes. In Section 4 we introduce
our main technical tool of this work. Section 5 is devoted to the
proof of Theorem \ref{THEOREM:01}, while Theorem \ref{THEOREM:03}
is proved in Section 6. Some additional sufficient conditions and
examples for condition (A1) are discussed in Section 7. Basic estimates
on stochastic integrals with respect to Lévy processes are discussed
in the appendix.

\section{Two special cases}

\subsection{The case of $\alpha$-stable laws}

We start with the particular case where $Z$ is given as in (Z1).
\begin{Remark} Suppose that the Lévy process $Z$ is given as in
(Z1). Then (A1) is satisfied for $\alpha=\alpha_{1}=\dots=\alpha_{d}$,
and \eqref{MOMENT:00} holds for any choice of $\gamma\in(\alpha,2]$
and $\delta\in(0,\alpha)$. Hence conditions \eqref{EQ:01} and \eqref{EQ:02} are reduced to: 
\begin{enumerate}
\item[(a)] If $\alpha\in[1,2)$, then \eqref{EQ:01} is automatically satisfied.
\item[(b)] If $\alpha\in(0,1)$, then \eqref{EQ:02} is implied by $\alpha+\beta\wedge\chi>1$.
This condition coincides with the one considered in \citep[Theorem 1.1]{DF13}.
\end{enumerate}
Hence we recover, for $Z$ given as in (Z1), the results obtained
in \citep{DF13}. \end{Remark} If $Z$ is given as in (Z2), we have
the following. \begin{Remark}\label{REMARK:10} Let $Z=(Z_{1},\dots,Z_{d})$
be a Lévy process given by (Z2).
\begin{enumerate}
\item[(a)] Suppose that (A2) is satisfied. Then \eqref{MOMENT:00}
holds for any $\gamma\in(\alpha^{\mathrm{max}},2]$
and $\delta\in(0,\alpha^{\mathrm{min}})$. Hence the assumptions
of Theorem \ref{THEOREM:01} are satisfied, provided:
\begin{enumerate}
\item[(i)] if $\alpha^{\mathrm{max}}\in[1,2)$, it holds that
\[
    \alpha^{\mathrm{min}} + \frac{\alpha^{\mathrm{min}}}{\alpha^{\mathrm{max}}} \left(\beta \wedge \alpha^{\mathrm{min}}\right) > 1, \qquad
    \frac{\alpha^{\mathrm{min}}}{\alpha^{\mathrm{max}}}\left( 1 + \chi \wedge \left( \frac{\alpha^{\mathrm{min}}}{\alpha^{\mathrm{max}}}\right) \right) > 1.
   \]

\item[(ii)] if $\alpha^{\mathrm{max}}\in(0,1)$, it holds that
\[
\frac{\alpha^{\mathrm{min}}}{\alpha^{\mathrm{max}}}+\alpha^{\mathrm{min}}\chi>1, \qquad
 \alpha^{\mathrm{min}}+ \frac{\alpha^{\mathrm{min}}}{\alpha^{\mathrm{max}}} (\beta\wedge\chi) > 1.
\]
\end{enumerate}
\item[(b)] Suppose that (A3) is satisfied and $\sigma$ is diagonal. Then \eqref{EQ:05}
is satisfied for $\gamma_{k}\in(\alpha_{k},2]$ and $\delta_{k}\in(0,\alpha_{k})$.
In particular, Theorem \ref{THEOREM:03} is applicable, provided for
each $k\in\{1,\dots,d\}$ with $\alpha_{k}\in(0,1)$, one has
\[
\alpha^{\mathrm{min}}+\min\{\chi_{1}\wedge\beta_{1},\dots,\chi_{d}\wedge\beta_{d}\}>1,\qquad\alpha_{k}+\frac{\alpha_{k}}{\alpha^{\mathrm{max}}}\left(\beta_{k}\wedge\chi_{k}\right)>1.
\]
Note that, for $k\in\{1,\dots,d\}$ and $\alpha_{k}\in[1,2)$, no
additional restriction on the parameters has to be imposed.
\end{enumerate}
\end{Remark} Such conditions provide reasonable restrictions on $\alpha_{1},\dots,\alpha_{d}$
and on the order of Hölder continuity. These conditions are more likely
to be satisfied in one of the following two cases:
\begin{enumerate}
\item[(i)] $\alpha^{\mathrm{min}}/\alpha^{\mathrm{max}}$ is large enough, i.e.,
$\alpha_{1},\dots,\alpha_{d}$ don't differ too much.
\item[(ii)] $\alpha^{\mathrm{min}}$ is large enough.
\end{enumerate}

\subsection{The case of zero drift}

In this part we briefly comment on some particular cases where the
drift vanishes. \begin{Remark} Suppose that $b=0$. By inspection
of the proofs (see Section 5 and Section 6) we easily deduce that
the restrictions imposed within Theorem \ref{THEOREM:01} and Theorem
\ref{THEOREM:03} can be slightly improved:
\begin{enumerate}
\item[(a)] Condition \eqref{EQ:01} can be replaced by
\begin{align*}
\frac{\alpha^{\mathrm{min}}}{\gamma}\left(1+\chi\wedge(\delta/\gamma)\right)>1,
\end{align*}
while \eqref{EQ:02} can be replaced by
\begin{align*}
\alpha^{\mathrm{min}}\left(\frac{1}{\gamma}+\chi\right)>1,\qquad\alpha^{\mathrm{min}}+\chi>1.
\end{align*}
\item[(b)] Condition \eqref{EQ:03} can be replaced by
\begin{align*}
\alpha_{k}\left(\frac{1}{\gamma_{k}}+\frac{\chi_{k}\wedge(\delta/\gamma_{k})}{\gamma}\right)>1,
\end{align*}
while \eqref{EQ:06} can be replaced by
\begin{align*}
\alpha^{\mathrm{min}}+\min\{\chi_{1},\dots,\chi_{d}\}>1,\qquad\alpha_{k}\left(\frac{1}{\gamma_{k}}+\frac{\chi_{k}}{\max\{1,\gamma\}}\right)>1.
\end{align*}
\end{enumerate}
\end{Remark} Finally, let us consider the particular case where $b=0$,
i.e.,
\begin{align}
dX(t)=\sigma(X(t))dZ(t),\ \ \ t>0,\label{EXAMPLE:00}
\end{align}
where $Z$ is given as in (Z2) and $\sigma$ is bounded and uniformly
elliptic. The following results for \eqref{EXAMPLE:00} are known:
\begin{enumerate}
\item[(i)] If $\sigma$ is continuous and $\alpha_{1}=\dots=\alpha_{d}$, then
existence and uniqueness in law was shown in \citep{BC06}.
\item[(ii)] If $\sigma$ is Hölder continuous, diagonal and $\alpha_{1}=\dots=\alpha_{d}$,
then the corresponding Markov process $X(t)$ has a transition density
for which also sharp two-sided estimates have been obtained (see \citep{KR17}).
\item[(iii)] If $\sigma$ is continuous and diagonal, but $\alpha_{1},\dots,\alpha_{d}$
may be different, then still existence and uniqueness in law holds,
see \citep{C18}.
\end{enumerate}
Below we obtain existence of densities to \eqref{EXAMPLE:00} also
applicable when $\alpha_{1},\dots,\alpha_{d}$ are different and $\sigma$
is not diagonal. 
\begin{Corollary}\label{THEOREM:00} 
Let $Z$ be given by (Z2). Suppose that $\sigma$ is bounded, uniformly elliptic, diagonal with
$\sigma = \mathrm{diag}(\sigma_1,\dots, \sigma_d)$, and there exists $\chi \in (0,1)$ such that $|\sigma(x) - \sigma(y)| \leq C |x-y|$ for all $x,y \in \R^d$.
 Finally,
\begin{enumerate}
\item[(a)] if $\alpha^{\mathrm{max}}\in[1,2)$, assume that $\frac{\alpha^{\mathrm{min}}}{\alpha^{\mathrm{max}}}>\frac{1}{1+\chi\wedge\left(\frac{\alpha^{\mathrm{min}}}{\alpha^{\mathrm{max}}}\right)}$;
\item[(b)] if $\alpha^{\mathrm{max}}\in(0,1)$, assume that $\alpha^{\mathrm{min}}>\frac{1}{1+\chi}$.
\end{enumerate}
Let $(X(t))_{t\geq0}$ satisfy \eqref{EQ:12} and denote by $\mu_{t}$
the law of $X_{t}$. Then there exists $\lambda\in(0,1)$ such that $\mu_{t}$ has a density $g_{t}\in B_{1,\infty}^{\lambda,a}(\R^{d})$
and this density satisfies \eqref{MAIN:03}. 
\end{Corollary} 
\begin{proof}
If $\alpha^{\mathrm{max}}\in(0,1)$, then $\alpha^{\mathrm{min}}>\frac{1}{1+\chi}$
implies that $\frac{\alpha^{\mathrm{min}}}{\alpha^{\mathrm{max}}}+\alpha^{\mathrm{min}}\chi>1$
and $\alpha^{\mathrm{min}}+\chi>1$ hold. The assertion follows from
Theorem \ref{THEOREM:01}. 
\end{proof} 
If $\sigma$ is, in addition,
diagonal, then we obtain the following. 
\begin{Corollary}
Let $Z$ be given by (Z2) and suppose that $\sigma$ is bounded, uniformly elliptic and diagonal with
$\sigma = \mathrm{diag}(\sigma_1,\dots, \sigma_d)$. Assume that each $\sigma_i$, $i = 1,\dots, d$, is $\chi$-H\"older continuous
for some $\chi \in (0,1)$. If $\alpha^{\mathrm{max}} \in (0,1)$ suppose, in addition, that
\[
 \alpha^{\mathrm{min}} + \min\{\chi_{1},\cdots,\chi_{d}\} > 1.
\]
Then there exists $\lambda\in(0,1)$ such that any solution $(X(t))_{t\geq0}$ to \eqref{EQ:12} has a
density $g_{t}\in B_{1,\infty}^{\lambda,a}(\R^{d})$ and this density
satisfies \eqref{MAIN:03}.
\end{Corollary}

\section{Main technical tool}
Existence of a density for solutions to \eqref{EQ:12} is, in the
isotropic case, essentially based on a discrete integration
by parts lemma formulated for the difference operator $\Delta_{h}$
acting on the classical Hölder-Zygmund space (see \citep[Lemma 2.1]{DF13}).
Such density belongs, by construction, to an isotropic Besov space.
In this work we use an anisotropic version of this lemma which is
designed for Lévy processes satisfying condition (A1).

The anisotropic Hölder-Zygmund space $C_{b}^{\lambda,a}(\R^{d})$ is defined as the Banach space of functions $\phi$ with finite norm
\[
\|\phi\|_{C_{b}^{\lambda,a}}=\|\phi\|_{\infty}+\sum\limits _{k=1}^{d}\sup\limits _{h\in[-1,1]}|h|^{-\lambda/a_{k}}\|\Delta_{he_{k}}\phi\|_{\infty}.
\]
The following is our main technical tool for the existence of a density.
\begin{Lemma}\label{LEMMA:02} Let $a=(a_{1},\dots,a_{d})$ be an
anisotropy in the sense of \eqref{EQ:35} and $\lambda,\eta>0$ be
such that $(\lambda+\eta)/a_{k}\in(0,1)$ for all $k=1,\dots,d$.
Suppose that $\mu$ is a finite measure over $\R^{d}$ and there exists
$A>0$ such that, for all $\phi\in C_{b}^{\eta,a}(\R^{d})$ and all
$k=1,\dots,d$,
\begin{align}
\left|\int\limits _{\R^{d}}(\phi(x+he_{k})-\phi(x))\mu(dx)\right|\leq A\|\phi\|_{C_{b}^{\eta,a}}|h|^{(\lambda+\eta)/a_{k}},\ \ \forall h\in[-1,1].\label{EQ:32}
\end{align}
Then $\mu$ has a density $g$ with respect to the Lebesgue measure and
\begin{align}
\|g\|_{B_{1,\infty}^{\lambda,a}}\leq\mu(\R^{d})+3dA(2d)^{\eta/\lambda}\left(1+\frac{\lambda}{\eta}\right)^{1+\frac{\eta}{\lambda}}.\label{MAIN:05}
\end{align}
\end{Lemma} \begin{proof} Given an anisotropy $a$ as above, define
the corresponding anisotropic (maximum-)norm on $\R^{d}$ by $|x|_{a}=\max\{|x_{1}|^{1/a_{1}},\dots,|x_{d}|^{1/a_{d}}\}$.
We show the assertion in 3 steps.

\textit{Step 1.} For $r\in(0,1]$ let $\varphi_{r}(x)=(2r)^{-d}\1_{\{|x|_{a}<r\}}=(2r)^{-d}\1_{\{|x_{1}|<r^{a_{1}}\}}\cdots\1_{\{|x_{d}|<r^{a_{d}}\}}$.
Fix $\psi\in L^{\infty}(\R^{d})$. Then for each $j$ and $h\in[-1,1]$
we have
\begin{align*}
|(\psi\ast\varphi_{r})(x+he_{j})-(\psi\ast\varphi_{r})(x)| & \leq\|\psi\|_{\infty}\int\limits _{\R^{d}}|\varphi_{r}(x+he_{j}-z)-\varphi_{r}(x-z)|dz\\
 & \leq\frac{\|\psi\|_{\infty}}{2r^{a_{j}}}\int\limits _{\R}|\1_{\{|x_{j}+h-z|<r^{a_{j}}\}}-\1_{\{|x_{j}-z|<r^{a_{j}}\}}|dz\\
 & \leq2\min\left\{ 1,\frac{|h|}{r^{a_{j}}}\right\} \|\psi\|_{\infty}\leq2r^{-\eta}|h|^{\eta/a_{j}}\|\psi\|_{\infty}.
\end{align*}
Since also $\|\psi\ast\varphi_{r}\|_{\infty}\leq\|\psi\|_{\infty}$,
we obtain $\|\psi\ast\varphi_{r}\|_{C_{b}^{\eta,a}}\leq3d\|\psi\|_{\infty}r^{-\eta}$
and hence
\begin{align*}
\left|\int\limits _{\R^{d}}\psi(x)[(\mu\ast\varphi_{r})(x+he_{j})-(\mu\ast\varphi_{r})(x)]dx\right| & =\left|\int\limits _{\R^{d}}\left[(\psi\ast\varphi_{r})(z-he_{j})-(\psi\ast\varphi_{r})(z)\right]\mu(dz)\right|\\
 & \leq A\|\psi\ast\varphi_{r}\|_{C_{b}^{\eta,a}}|h|^{(\eta+\lambda)/a_{j}}\\
 & \leq3dA\|\psi\|_{\infty}r^{-\eta}|h|^{(\eta+\lambda)/a_{j}}.
\end{align*}
By duality this implies that for all $j=1,\dots,d$ and $h\in[-1,1]$,
\begin{align}
\int\limits _{\R^{d}}|(\mu\ast\varphi_{r})(x+he_{j})-(\mu\ast\varphi_{r})(x)|dx\leq3dAr^{-\eta}|h|^{(\eta+\lambda)/a_{j}}.\label{PROOF:LEMMA02:00}
\end{align}
\textit{Step 2.} Let us first suppose that $\mu$ has a density $g\in C^{1}(\R^{d})$
with $\nabla g\in L^{1}(\R^{d})$. Then we obtain, for any $h\in\R^{d}$
with $|h|_{a}\leq1$,
\begin{align*}
\int\limits _{\R^{d}}|g(x+h)-g(x)|dx & \leq\sum\limits _{k=1}^{d}\int\limits _{\R^{d}}|g(x+h_{k}e_{k})-g(x)|dx\\
 & \leq\sum\limits _{k=1}^{d}\int\limits _{\R^{d}}|(g\ast\varphi_{r})(x+h_{k}e_{k})-(g\ast\varphi_{r})(x)|dx+2d\int\limits _{\R^{d}}|(g\ast\varphi_{r})(x)-g(x)|dx\\
 & \leq3dAr^{-\eta}\sum\limits _{k=1}^{d}|h_{k}|^{(\eta+\lambda)/a_{k}}+\frac{2d}{(2r)^{d}}\int\limits _{\R^{d}}\int\limits _{\R^{d}}|g(y)-g(x)|\1_{\{|x-y|_{a}<r\}}dydx\\
 & \leq3d^{2}Ar^{-\eta}|h|_{a}^{\eta+\lambda}+\frac{2d}{(2r)^{d}}\int\limits _{\{|u|_{a}<r\}}\int\limits _{\R^{d}}|g(x+u)-g(x)|dxdu,
\end{align*}
where we have used \eqref{PROOF:LEMMA02:00}. For $t\in(0,1]$, set
$I_{t}:=\sup_{|h|_{a}=t}\int_{\R^{d}}|g(x+h)-g(x)|dx$ and $S_{t}:=\sup_{s\in(0,t]}s^{-\lambda}I_{s}$.
Then, using $\int_{\R^{d}}|u|_{a}^{\lambda}\1_{\{|u|_{a}<r\}}du\leq(2r)^{d}r^{\lambda}$
and $S_{|u|_{a}}\leq S_{1}$ for $|u|_{a}\leq1$, we obtain
\begin{align*}
t^{-\lambda}I_{t} & \leq3d^{2}A\left(\frac{t}{r}\right)^{\eta}+\frac{2d}{(2r)^{d}}t^{-\lambda}\int\limits _{\R^{d}}|u|_{a}^{\lambda}\1_{\{|u|_{a}<r\}}S_{|u|_{a}}du\\
 & \leq3d^{2}A\left(\frac{t}{r}\right)^{\eta}+2d\left(\frac{r}{t}\right)^{\lambda}S_{1}.
\end{align*}
Letting $r=\e t$ with $0<\e^{\lambda}<(2d)^{-1}$ and then taking
the supremum over $t\in(0,1]$ on both sides yield $S_{1}\leq3d^{2}A\e^{-\eta}+2d\e^{\lambda}S_{1}$
and hence $S_{1}\leq\frac{3d^{2}A}{1-2d\e^{\lambda}}\frac{1}{\e^{\eta}}$.
A simple extreme value analysis shows that the right-hand side of
the last inequality attains its minimum for $\e^{\lambda}=\frac{\eta}{\lambda+\eta}\frac{1}{2d}$,
which gives
\[
\sup\limits _{|h|_{a}\leq1}|h|_{a}^{-\lambda}\int\limits _{\R^{d}}|g(x+h)-g(x)|dx=S_{1}\leq3d^{2}A(2d)^{\eta/\lambda}\left(1+\frac{\lambda}{\eta}\right)^{1+\frac{\eta}{\lambda}}.
\]
\textit{Step 3.} For the general case let $\mu_{n}=\mu\ast G_{n}$,
where $G_{n}(x)=n^{d}G(nx)$ and $0\leq G\in C_{c}^{\infty}(\R^{d})$
satisfies $\|G\|_{L^{1}(\R^{d})}=1$ with support in $\{x\in\R^{d}\ |\ |x|\leq1\}$.
Then $\mu_{n}$ has a density $g_{n}\in C^{1}(\R^{d})$ with $\nabla g_{n}\in L^{1}(\R^{d})$.
Moreover, for any $\phi\in C_{b}^{\eta,a}(\R^{d})$ and $h\in[-1,1]$,
\begin{align*}
\left|\int\limits _{\R^{d}}(\phi(x+he_{j})-\phi(x))\mu_{n}(dx)\right| & =\left|\int\limits _{\R^{d}}[(\phi\ast G_{n})(x+he_{j})-(\phi\ast G_{n})(x)]\mu(dx)\right|\\
 & \leq A\|\phi\ast G_{n}\|_{C_{b}^{\eta,a}}|h|^{(\lambda+\eta)/a_{j}}\leq A\|\phi\|_{C_{b}^{\eta,a}}|h|^{(\lambda+\eta)/a_{j}}.
\end{align*}
Applying step 2 for $g_{n}$ gives
\begin{align}
\|g_{n}\|_{B_{1,\infty}^{\lambda,a}} & \leq\mu(\R^{d})+\frac{1}{d}\sup\limits _{|h|_{a}\leq1}|h|_{a}^{-\lambda}\int\limits _{\R^{d}}|g_{n}(x+h)-g_{n}(x)|dx\label{MAIN:06}\\
 & \leq\mu(\R^{d})+3dA(2d)^{\eta/\lambda}\left(1+\frac{\lambda}{\eta}\right)^{1+\frac{\eta}{\lambda}}.\nonumber
\end{align}
The particular choice of $G_{n}$ implies that $g_{n}$ is also tight,
i.e.
\[
\sup\limits _{n\in\N}\int\limits _{|x|\geq R}g_{n}(x)dx\longrightarrow0,\ \ R\to\infty.
\]
Hence we may apply the Kolmogorov-Riesz compactness criterion for
$L^{1}(\R^{d})$ to conclude that $(g_{n})_{n\in\N}$ is relatively
compact in $L^{1}(\R^{d})$. Let $g\in L^{1}(\R^{d})$ be the limit
of a subsequence of $(g_{n})_{n\in\N}$. Since also $g_{n}(x)dx\longrightarrow\mu(dx)$
weakly as $n\to\infty$, we conclude that $\mu(dx)=g(x)dx$. The estimate
\eqref{MAIN:05} is now a consequence of \eqref{MAIN:06} and the
Lemma of Fatou. \end{proof} Note that the restriction $(\lambda+\eta)/a_{k}\in(0,1)$,
$k=1,\dots,d$, in the above lemma is not essential
since we may always replace $\lambda,\eta>0$ by some smaller values
which satisfy this condition and \eqref{EQ:32}.

\section{Proof of Theorem \ref{THEOREM:01}}

\subsection{Approximation when $\gamma\in[1,2]$}

Suppose that $\gamma\in[1,2]$, and define
\begin{align}
\kappa=\kappa(\gamma,\delta,\beta,\chi)=\min\left\{ 1+\frac{\beta\wedge\delta}{\gamma},\frac{1}{\gamma}+\frac{\chi\wedge(\delta/\gamma)}{\gamma}\right\} .\label{KAPPA:00}
\end{align}
Let $X$ be given by \eqref{EQ:12} and define, for $t>0$ and $\e\in(0,1\wedge t)$,
\begin{align*}
X^{\e}(t)=U^{\e}(t)+\sigma(X(t-\e))(Z(t)-Z(t-\e)),\qquad U^{\e}(t)=X(t-\e)+b(X(t-\e))\e.
\end{align*}
Below we provide an explicit convergence rate for $X^{\e}\to X$,
when $\e\to0$. A similar result was obtained in the isotropic case
\citep[Lemma 3.1]{DF13}. \begin{Proposition}\label{PROP:00} Let
$Z$ be a Lévy process with Lévy measure $\nu$ and symbol \eqref{SYMBOL}.
Suppose that \eqref{MOMENT:00} holds with $\gamma\in[1,2]$ and condition
(A2) is satisfied. Let $X$ be given as in \eqref{EQ:12}. Then
\begin{enumerate}
\item[(a)] For each $\eta\in(0,\delta]$ there exists a constant $C=C(\eta)>0$
such that, for all $0\leq s\leq t\leq s+1$, it holds that
\[
\E[|X(t)-X(s)|^{\eta}]\leq C(t-s)^{\eta/\gamma}.
\]
\item[(b)] For each $\eta\in(0,1\wedge\delta]$ there exists a constant $C=C(\eta)>0$
such that
\[
\E[|X(t)-X^{\e}(t)|^{\eta}]\leq C\e^{\eta\kappa},\ \ t>0,\ \ \e\in(0,1\wedge t).
\]
\end{enumerate}
\end{Proposition} \begin{proof} \textit{(a)} Write
\[
\E[|X(t)-X(s)|^{\eta}]\leq C\E\left[\left|\int\limits _{s}^{t}b(X(u))du\right|^{\eta}\right]+C\E\left[\left|\int\limits _{s}^{t}\sigma(X(u))dZ(u)\right|^{\eta}\right]=:R_{1}+R_{2}.
\]
Using the boundedness of $b$ and $\eta\leq\gamma$ yields for the
first term $R_{1}\leq C(t-s)^{\eta}\leq C(t-s)^{\eta/\gamma}$. For
the second term we apply Lemma \ref{LEMMA:00}.(a) and the boundedness
of $\sigma$ to obtain
\begin{align*}
R_{2}\leq C(t-s)^{\eta/\gamma}\sup\limits _{u\in[s,t]}\E[|\sigma(X(u))|^{\gamma}]^{\eta/\gamma}\leq C(t-s)^{\eta/\gamma}.
\end{align*}
This proves the assertion.

\textit{(b)} Write $\E[|X(t)-X^{\e}(t)|^{\eta}]\leq R_{1}+R_{2}$,
where
\begin{align*}
R_{1} & =\E\left[\left|\int\limits _{t-\e}^{t}(b(X(u))-b(X(t-\e)))du\right|^{\eta}\right],\\
R_{2} & =\E\left[\left|\int\limits _{t-\e}^{t}(\sigma(X(t-))-\sigma(X(t-\e)))dZ(u)\right|^{\eta}\right].
\end{align*}
Let us first estimate $R_{1}$. From (A2) we obtain
$|b(x)-b(y)|\leq C|x-y|^{\beta\wedge\delta}$. Applying first Jensen
inequality and then \textit{(a)} gives
\begin{align*}
R_{1} & \leq\E\left[\int\limits _{t-\e}^{t}|b(X(u))-b(X(t-\e))|du\right]^{\eta}\\
 & \leq C\e^{\eta}\sup\limits _{u\in[t-\e,t]}\E\left[|X(u)-X(t-\e)|^{\beta\wedge\delta}\right]^{\eta}\leq C\e^{\eta+\frac{\eta(\beta\wedge\delta)}{\gamma}}.
\end{align*}
For the second term we use Lemma \ref{LEMMA:00}.(a), $|\sigma(x)-\sigma(y)|\leq C|x-y|^{\chi\wedge(\delta/\gamma)}$
(see (A2)) and then part \textit{(a)} to obtain
\begin{align*}
R_{2} & \leq C\e^{\eta/\gamma}\sup\limits _{u\in[t-\e,t]}\E\left[|\sigma(X(u))-\sigma(X(t-\e))|^{\gamma}\right]^{\eta/\gamma}\\
 & \leq C\e^{\eta/\gamma}\sup\limits _{u\in[t-\e,t]}\E[|X(u)-X(t-\e)|^{\gamma(\chi\wedge(\delta/\gamma))}]^{\eta/\gamma}\leq C\e^{\frac{\eta}{\gamma}+\frac{\eta}{\gamma}\chi\wedge(\delta/\gamma)}.
\end{align*}
This proves the assertion. \end{proof}

\subsection{Approximation when $\gamma\in(0,1)$}

In this part we study the case when $\gamma\in(0,1)$. Our aim is
to provide a similar approximation as in the case $\gamma\in[1,2]$.
In order to obtain an optimal convergence rate $\kappa$, we use the
same ideas as in \citep[Lemma 3.2]{DF13}. Define, for $\gamma\in(0,1)$,
\begin{align}
\kappa=\kappa(\gamma,\beta,\chi)=\min\left\{ 1+\frac{\beta\wedge\chi}{\gamma},\frac{1}{\gamma}+\chi,\frac{1}{1-\beta\wedge\chi}\right\} .\label{KAPPA:01}
\end{align}
Below we provide a similar result to Proposition \ref{PROP:00}. \begin{Proposition}\label{PROP:04}
Let $Z$ be a Lévy process with Lévy measure $\nu$ and symbol \eqref{SYMBOL}.
Suppose that \eqref{MOMENT:00} holds for $\gamma\in(0,1)$ and condition
(A2) is satisfied. Let $X$ be as in \eqref{EQ:12}. Then
\begin{enumerate}
\item[(a)] For each $\eta\in(0,\delta]$ there exists a constant $C>0$ such
that, for all $0\leq s\leq t\leq s+1$, it holds that
\[
\E[|X(t)-X(s)|^{\eta}]\leq C(t-s)^{\eta}.
\]
Moreover, if $\eta\in(0,1)$, then
\[
\E[|X(t)-X(s)|^{\eta}\wedge1]\leq C(t-s)^{\eta}.
\]
\item[(b)] For each $t>0$ and $\e\in(0,1\wedge t)$ there exists an $\F_{t-\e}$-measurable random variable $U^{\e}(t)$ such that,
setting
\[
X^{\e}(t)=U^{\e}(t)+\sigma(X(t-\e))(Z(t)-Z(t-\e)),
\]
for any $\eta\in(0,\delta]$ there exists a constant $C>0$ with
\[
\E\left[|X(t)-X^{\e}(t)|^{\eta}\right]\leq C\e^{\eta\kappa},\ \ t>0,\ \ \e\in(0,1\wedge t).
\]
\end{enumerate}
\end{Proposition} \begin{proof} \textit{(a)} Suppose that $\eta\in(0,\delta]$.
Then $\E[|X(t)-X(s)|^{\eta}\wedge1]\leq\E[|X(t)-X(s)|^{\eta}]$ and
\[
\E[|X(t)-X(s)|^{\eta}]\leq C\E\left[\left|\int\limits _{s}^{t}b(X(u))du\right|^{\eta}\right]+C\E\left[\left|\int\limits _{s}^{t}\sigma(X(u-))dZ(u)\right|^{\eta}\right]=:R_{1}+R_{2}.
\]
Using the Jensen inequality and the boundedness of $b$, we obtain
$R_{1}\leq C(t-s)^{\eta}$. For the second term we let $Y(t):=Z(t)+t\int_{|z|\leq1}z\nu(dz)$
and obtain from Lemma \ref{LEMMA:00}.(b)
\begin{align*}
R_{2} & \leq C\E\left[\left|\int\limits _{s}^{t}\sigma(X(s-))dY(s)\right|^{\eta}\right]+C\E\left[\left|\int\limits _{s}^{t}\sigma(X(s))ds\right|^{\eta}\right]\\
 & \leq C(t-s)^{\eta/\gamma}\sup\limits _{u\in[s,t]}\E[|\sigma(X(u))|^{\gamma}]^{\eta/\gamma}+C(t-s)^{\eta}\\
 & \leq C(t-s)^{\eta},
\end{align*}
where in the last inequality we have used that $\gamma\in(0,1)$.
Suppose that $\eta\in(\delta,1)$ and let $\widetilde{b}(x)=b(x)-\sigma(x)\int_{|z|\leq1}z\nu(dz)$.
Using the boundedness of $\widetilde{b}$, we obtain
\begin{align*}
|X(t)-X(s)|^{\eta}\wedge1 & \leq\left|\int\limits _{s}^{t}\widetilde{b}(X(u))du\right|^{\eta}\wedge1+\left|\int\limits _{s}^{t}\sigma(X(u-))dY(u)\right|^{\eta}\\
 & \leq C(t-s)^{\eta}+\left|\int\limits _{s}^{t}\sigma(X(u-))dY(u)\right|^{\eta}.
\end{align*}
Finally, using Lemma \ref{LEMMA:00}.(b), we conclude that
\[
\E\left[\left|\int\limits _{s}^{t}\sigma(X(u-))dY(u)\right|^{\eta}\right]\leq C(t-s)^{\eta/\gamma}\leq C(t-s)^{\eta},
\]
which proves the assertion.

\textit{(b)} Set $\tau=\e^{1/(1-\beta\wedge\chi)}$ and define, for
$s\in[t-\e,t]$, $s_{\tau}=t-\e+\tau\lfloor(s-(t-\e))/\tau\rfloor$.
Let $W^{\e}$ be the solution to
\[
W^{\e}(u)=X(t-\e)+\int\limits _{t-\e}^{u}\widetilde{b}(W^{\e}(s_{\tau}))ds,\ \ u\in[t-\e,t].
\]
Arguing as in \citep[Lemma 3.2]{DF13}, we see that $W^{\e}(t)$ is
well-defined and $\F_{t-\e}$-measurable (because it is a deterministic
continuous function of $X(t-\e)$). Define
\begin{align}
X^{\e}(t) & =W^{\e}(t)+\sigma(X(t-\e))(Y(t)-Y(t-\e))=U^{\e}(t)+\sigma(X(t-\e))(Z(t)-Z(t-\e)),\label{APPROXIMATION:02}\\
U^{\e}(t) & =W^{\e}(t)+\e\sigma(X(t-\e))\int_{|z|\leq1}z\nu(dz).\nonumber
\end{align}
It remains to prove the desired estimate. Thus write
\begin{align*}
X^{\e}(t)=X(t-\e)+\int\limits _{t-\e}^{t}\widetilde{b}(W^{\e}(s))ds+\int\limits _{t-\e}^{t}\sigma(X(t-\e))dY(s)+\int\limits _{t-\e}^{t}\left(\widetilde{b}(W^{\e}(s_{\tau}))-\widetilde{b}(W^{\e}(s))\right)ds,
\end{align*}
so that $\E[|X(t)-X^{\e}(t)|^{\eta}\leq\E[I_{\e}^{\eta}]+\E[J_{\e}^{\eta}]+\E[J_{\e}^{\eta}]$
with
\begin{align*}
I_{\e} & =\int\limits _{t-\e}^{t}|\widetilde{b}(X(s))-\widetilde{b}(W^{\e}(s))|ds,\\
J_{\e} & =\left|\int\limits _{t-\e}^{t}(\sigma(X(s-))-\sigma(X(t-\e)))dY(s)\right|,\\
K_{\e} & =\int\limits _{t-\e}^{t}|\widetilde{b}(W^{\e}(s_{\tau}))-\widetilde{b}(W^{\e}(s))|ds.
\end{align*}
For the second term we apply Lemma \ref{LEMMA:00}.(b), the Hölder
continuity of $\sigma$ and part \textit{(a)}, so that
\begin{align*}
\E[J_{\e}^{\eta}] & \leq C\e^{\eta/\gamma}\sup\limits _{s\in[t-\e,t]}\E[|\sigma(X(s))-\sigma(X(t-\e))|^{\gamma}]^{\eta/\gamma}\\
 & \leq C\e^{\eta/\gamma}\sup\limits _{s\in[t-\e,t]}\E[|X(s)-X(t-\e)|^{\gamma\chi}\wedge1]^{\eta/\gamma}\leq C\e^{\eta\left(\frac{1}{\gamma}+\chi\right)}.
\end{align*}
For the last term we use the $(\beta\wedge\chi)$-Hölder continuity
of $\widetilde{b}$ to obtain
\begin{align*}
\E[K_{\e}^{\eta}]\leq C\E\left[\int\limits _{t-\e}^{t}|W^{\e}(s_{\tau})-W^{\e}(s)|^{\beta\wedge\chi}ds\right]^{\eta}\leq C\e\tau^{\eta(\beta\wedge\chi)}=C\e^{\frac{\eta}{1-\beta\wedge\chi}},
\end{align*}
where we have used $|W^{\e}(s)-W^{\e}(s_{\tau})|\leq C\tau$ (since
$\widetilde{b}$ is bounded). The assertion is proved, if we can show
that
\[
\E[I_{\e}^{\eta}]\leq C\left(\e^{\frac{\eta}{1-\beta\wedge\chi}}+\e^{\eta+\eta\frac{\beta\wedge\chi}{\gamma}}\right).
\]
In order to prove this estimate we proceed as follows. Write
\[
W^{\e}(u)=X(t-\e)+\int\limits _{t-\e}^{u}\widetilde{b}(W^{\e}(s))ds+\int\limits _{t-\e}^{u}\left(\widetilde{b}(W^{\e}(s_{\tau}))-\widetilde{b}(W^{\e}(s))\right)ds,
\]
let $R^{\e}(t)=\sum_{s\in[t-\e,t]}|\Delta Y(s)|$, with $\Delta Y(s)=\lim_{r\nearrow s}(Y(s)-Y(r))$,
so that, for $u\in[t-\e,t]$,
\begin{align*}
|X(u)-W^{\e}(u)|\leq CR^{\e}(t)+\int\limits _{t-\e}^{u}|\widetilde{b}(X(s))-\widetilde{b}(W^{\e}(s))|ds+\int\limits _{t-\e}^{u}|\widetilde{b}(W^{\e}(s_{\tau}))-\widetilde{b}(W^{\e}(s))|ds.
\end{align*}
Setting $S^{\e}(t)=\sup_{s\in[t-\e,t]}|X(s)-W^{\e}(s)|$, using the
$(\beta\wedge\chi)$-Hölder continuity of $\widetilde{b}$ we obtain
\begin{align*}
S^{\e}(t) & \leq C\left(R^{\e}(t)+\e S^{\e}(t)^{\beta\wedge\chi}+\e\tau^{\beta\wedge\chi}\right)\\
 & =C\left(R^{\e}(t)+\e S^{\e}(t)^{\beta\wedge\chi}+\e^{\frac{1}{1-\beta\wedge\chi}}\right)\\
 & \leq CR^{\e}(t)+C\e^{1/(1-\beta\wedge\chi)}+(\beta\wedge\chi)S^{\e}(t),
\end{align*}
where we have used the Young inequality $xy\leq(1-\beta\wedge\chi)x^{1/(1-\beta\wedge\chi)}+\beta\wedge\chi y^{1/\beta\wedge\chi}$
with $x=C\e$ and $y=S^{\e}(t)^{\beta\wedge\chi}$. Since $\beta\wedge\chi<1$
we obtain
\begin{align}
S^{\e}(t)\leq CR^{\e}(t)+C\e^{\frac{1}{1-\beta\wedge\chi}}.\label{APPROXIMATION:00}
\end{align}
Estimating $I_{\e}$ gives, by \eqref{APPROXIMATION:00} and the $\beta\wedge\chi$-Hölder
continuity of $\widetilde{b}$,
\begin{align*}
I_{\e}\leq C\int_{t-\e}^{t}|X(s)-W^{\e}(s)|^{\beta\wedge\chi}ds\leq C\e S^{\e}(t)^{\beta\wedge\chi}\leq C\e\left(R^{\e}(t)^{\beta\wedge\chi}+\e^{\frac{\beta\wedge\chi}{1-\beta\wedge\chi}}\right).
\end{align*}
Using now Lemma \ref{LEMMA:00}.(b) we obtain
\[
\E[I_{\e}^{\eta}]\leq C\e^{\eta}\left(\e^{\eta\frac{\beta\wedge\chi}{1-\beta\wedge\chi}}+\E[R^{\e}(t)^{\eta(\beta\wedge\chi)}]\right)\leq C\left(\e^{\frac{\eta}{1-\beta\wedge\chi}}+\e^{\eta+\eta\frac{\beta\wedge\chi}{\gamma}}\right).
\]
This completes the proof. \end{proof}

\subsection{Main technical estimate}

Recall that $\kappa$ is defined by \eqref{KAPPA:00} or \eqref{KAPPA:01},
respectively. Based on the previous approximation we may show the
following. \begin{Proposition}\label{PROP:01} Let $Z$ be a Lévy
process with Lévy measure $\nu$ and symbol \eqref{SYMBOL}. Suppose
that \eqref{MOMENT:00}, (A1) and (A2) are satisfied. Let $X$ be
as in \eqref{EQ:12}, take an anisotropy $a=(a_{i})_{i\in\{1,\dots,d\}}$
and $\eta\in(0,1)$ with
\begin{align}
\frac{\eta}{a_{j}}\leq1\wedge\delta,\ \ j\in\{1,\dots,d\}.\label{EQ:00}
\end{align}
Then there exists a constant $C=C(\eta)>0$ and $\e_{0}\in(0,1\wedge t)$
such that, for any $\e\in(0,\e_{0})$, $h\in[-1,1]$, $\phi\in C_{b}^{\eta,a}(\R^{d})$
and $i\in\{1,\dots,d\}$,
\[
\left|\E\left[|\sigma^{-1}(X(t))|^{-1}\Delta_{he_{i}}\phi(X(t))\right]\right|\leq C\|\phi\|_{C_{b}^{\eta,a}}\left(|h|^{\eta/a_{i}}\e^{\frac{\chi\wedge\delta}{\max\{1,\gamma\}}}+|h|\e^{-1/\alpha_{i}}+\max\limits _{j\in\{1,\dots,d\}}\e^{\eta\kappa/a_{j}}\right).
\]
\end{Proposition} \begin{proof} For $\e\in(0,1\wedge t)$ let $X^{\e}(t)$
be the approximation from Proposition \ref{PROP:00} or Proposition
\ref{PROP:04}, respectively. Then
\begin{align*}
\left|\E\left[|\sigma^{-1}(X(t))|^{-1}\Delta_{he_{i}}\phi(X(t))\right]\right| & \leq R_{1}+R_{2}+R_{3},\\
R_{1} & =\left|\E\left[\Delta_{he_{i}}\phi(X(t))\left(|\sigma(X(t))|^{-1}-|\sigma^{-1}(X(t-\e))|^{-1}\right)\right]\right|,\\
R_{2} & =\E\left[|\Delta_{he_{i}}\phi(X(t))-\Delta_{he_{i}}\phi(X^{\e}(t))||\sigma^{-1}(X(t-\e))|^{-1}\right],\\
R_{3} & =\left|\E\left[|\sigma^{-1}(X(t-\e))|^{-1}\Delta_{he_{i}}\phi(X^{\e}(t))\right]\right|.
\end{align*}
For the first term we use \eqref{MATRIX:INEQ} to get $||\sigma^{-1}(x)|^{-1}-|\sigma^{-1}(y)|^{-1}|\leq|\sigma(x)-\sigma(y)|\leq C|x-y|^{\chi\wedge\delta}$
and then Proposition \ref{PROP:00}.(a) or Propostion \ref{PROP:04}.(a),
respectively, to obtain
\begin{align*}
R_{1}\leq\|\phi\|_{C_{b}^{\eta,a}}|h|^{\eta/a_{i}}\E[|X(t)-X(t-\e)|^{\chi\wedge\delta}|]\leq C\|\phi\|_{C_{b}^{\eta,a}}|h|^{\eta/a_{i}}\e^{\frac{\chi\wedge\delta}{\max\{1,\gamma\}}}.
\end{align*}
For $R_{2}$ we use again \eqref{MATRIX:INEQ}, i.e. $|\sigma^{-1}(x)|^{-1}\leq|\sigma(x)|\leq C$,
and Proposition \ref{PROP:00} to obtain
\begin{align*}
R_{2} & \leq C\|\phi\|_{C_{b}^{\eta,a}}\max\limits _{j\in\{1,\dots,d\}}\E\left[|\sigma^{-1}(X(t-\e))|^{-1}|X_{j}(t)-X_{j}^{\e}(t)|^{\eta/a_{j}}\right]\\
 & \leq C\|\phi\|_{C_{b}^{\eta,a}}\max\limits _{j\in\{1,\dots,d\}}\E\left[|X_{j}(t)-X_{j}^{\e}(t)|^{\frac{\eta}{a_{j}}}\right]\\
 & \leq C\|\phi\|_{C_{b}^{\eta,a}}\max\limits _{j\in\{1,\dots,d\}}\e^{\eta\kappa/a_{j}},
\end{align*}
where in the last inequality we have used \eqref{EQ:00} so that Proposition
\ref{PROP:00}.(b) or Proposition \ref{PROP:04}.(b) is applicable.
Let us turn to $R_{3}$. Let $f_{t}$ be the density given by (A1)
and write $X^{\e}(t)=U^{\e}(t)+\sigma(X(t-\e))(Z(t)-Z(t-\e))$, where
$U^{\e}(t)$ is either given by Proposition \ref{PROP:00}.(b) or
Proposition \ref{PROP:04}.(b), respectively. Then we obtain
\begin{align*}
R_{3} & =\left|\E\left[\int\limits _{\R^{d}}|\sigma^{-1}(X(t-\e))|^{-1}(\Delta_{he_{i}}\phi)(U^{\e}(t)+\sigma(X(t-\e))z)f_{\e}(z)dz\right]\right|\\
 & =\left|\E\left[\int\limits _{\R^{d}}|\sigma^{-1}(X(t-\e))|^{-1}\phi(U^{\e}(t)+\sigma(X(t-\e))z)(\Delta_{-h\sigma^{-1}(X(t-\e))e_{i}}f_{\e})(z)dz\right]\right|\\
 & \leq\|\phi\|_{\infty}\E\left[|\sigma^{-1}(X(t-\e))|^{-1}\int\limits _{\R^{d}}|(\Delta_{-h\sigma^{-1}(X(t-\e))e_{i}}f_{\e})(z)|dz\right]\\
 & \leq C\|\phi\|_{C_{b}^{\eta,a}}|h|\e^{-1/\alpha_{i}}\E\left[|\sigma^{-1}(X(t-\e))|^{-1}|\sigma^{-1}(X(t-\e))e_{i}|\right]\\
 & \leq C\|\phi\|_{C_{b}^{\eta,a}}|h|\e^{-1/\alpha_{i}},
\end{align*}
where we have used \eqref{EXAMPLE:10} for $\e\in(0,\e_{0})$, $\e_{0}\in(0,1)$
small enough and,
\[
\E\left[|\sigma^{-1}(X(t-\e))|^{-1}|\sigma^{-1}(X(t-\e))e_{i}|\right]\leq1.
\]
Summing up the estimates for $R_{0},R_{1},R_{2},R_{3}$ yields the
assertion. \end{proof}

\subsection{Concluding the proof of Theorem \ref{THEOREM:01}}

Below we provide the proof of Theorem \ref{THEOREM:01}. Fix $t>0$.
It suffices to show that Lemma \ref{LEMMA:02} is applicable to the
finite measure
\[
\mu_{t}(A)=\E[|\sigma^{-1}(X(t))|^{-1}\1_{A}(X(t))],\ \ A\subset\R^{d}\ \ \text{Borel}.
\]
Using \eqref{EQ:01} or \eqref{EQ:02}, respectively, we obtain $\kappa\alpha_{j}>1$
and hence $\kappa/a_{j}>1/\overline{\alpha}$ for all $j\in\{1,\dots,d\}$.
This implies
\[
\frac{a_{j}}{\kappa}\frac{1}{a_{i}}<\frac{\overline{\alpha}}{a_{i}}=\alpha_{i},\ \ i,j\in\{1,\dots,d\}.
\]
Hence we find $\eta\in(0,1)$ and $c_{1},\dots,c_{d}>0$ such that,
for all $i,j\in\{1,\dots,d\}$,
\[
0<\frac{\eta}{a_{i}}<1\wedge\delta,\qquad\frac{a_{j}}{\kappa}\frac{1}{a_{i}}<c_{i}<\alpha_{i}\left(1-\frac{\eta}{a_{i}}\right).
\]
Define
\[
\lambda=\min_{i,j\in\{1,\dots,d\}}\left\{ c_{i}\frac{\chi\wedge\delta}{\max\{1,\gamma\}}a_{i},\ a_{i}-\eta-\frac{a_{i}c_{i}}{\alpha_{i}},\ \eta\left(c_{i}a_{i}\frac{\kappa}{a_{j}}-1\right)\right\} >0.
\]
Let $\phi\in C_{b}^{\eta,a}(\R^{d})$. By Proposition \ref{PROP:01}
we obtain, for $h\in[-1,1]$, $\e=|h|^{c_{i}}(1\wedge t)$ and $i\in\{1,\dots,d\}$,
\begin{align*}
 & \ \left|\E\left[|\sigma^{-1}(X(t))|^{-1}\Delta_{he_{i}}\phi(X(t))\right]\right|\\
 & \leq C\|\phi\|_{C_{b}^{\eta,a}}\left(|h|^{\eta/a_{i}}\e^{\frac{\chi\wedge\delta}{\max\{1,\gamma\}}}+|h|\e^{-1/\alpha_{i}}+\max\limits _{j\in\{1,\dots,d\}}\e^{\eta\kappa/a_{j}}\right)\\
 & \leq\frac{C\|\phi\|_{C_{b}^{\eta,a}}}{(1\wedge t)^{1/\alpha_{i}}}\left(|h|^{\eta/a_{i}+c_{i}\frac{\chi\wedge\delta}{\max\{1,\gamma\}}}+|h|^{1-c_{i}/\alpha_{i}}+\max\limits _{j\in\{1,\dots,d\}}|h|^{c_{i}\eta\kappa/a_{j}}\right)\\
 & =\frac{C\|\phi\|_{C_{b}^{\eta,a}}}{(1\wedge t)^{1/\alpha_{i}}}|h|^{\eta/a_{i}}\left(|h|^{c_{i}\frac{\chi\wedge\delta}{\max\{1,\gamma\}}}+|h|^{1-\eta/a_{i}-c_{i}/\alpha_{i}}+\max\limits _{j\in\{1,\dots,d\}}|h|^{c_{i}\eta\kappa/a_{j}-\eta/a_{i}}\right)\\
 & \leq\frac{C\|\phi\|_{C_{b}^{\eta,a}}}{(1\wedge t)^{1/\alpha_{i}}}|h|^{(\eta+\lambda)/a_{i}}.
\end{align*}
This shows that Lemma \ref{LEMMA:02} is applicable.

\section{Proof of Theorem \ref{THEOREM:02}}

\subsection{Approximation when $\gamma_{i}\in[1,2]$}

Recall that $\gamma=\max\{\gamma_{1},\dots,\gamma_{d}\}$ and $\delta=\min\{\delta_{1},\dots,\delta_{d}\}$.
For $i\in\{1,\dots,d\}$ such that $\gamma_{i}\in[1,2]$, define
\[
\kappa_{i}=\min\left\{ 1+\frac{\beta_{i}\wedge\delta}{\gamma},\frac{1}{\gamma_{i}}+\frac{\chi_{i}\wedge(\delta/\gamma_{i})}{\gamma}\right\} .
\]
Let $X$ be given by \eqref{EQ:12}. For $t>0$ and $\e\in(0,1\wedge t)$
define
\begin{align*}
X_{i}^{\e}(t)=U_{i}^{\e}(t)+\sigma_{i}(X(t-\e))(Z_{i}(t)-Z_{i}(t-\e)),\qquad U_{i}^{\e}(t)=X_{i}(t-\e)+b_{i}(X(t-\e))\e.
\end{align*}
Below we provide a similar convergence rate to Proposition \ref{PROP:00}.
\begin{Proposition}\label{PROP:05} 
Let $Z$ be a Lévy process with
Lévy measure $\nu$ and symbol \eqref{SYMBOL}. Suppose that \eqref{EQ:05}
and (A3) are satisfied. Let $i\in\{1,\dots,d\}$ be such that $\gamma_{i}\in[1,2]$,
and let $X$ be given as in \eqref{EQ:12}. Then
\begin{enumerate}
\item[(a)] For each $\eta\in(0,\delta_{i}]$ there exists a constant $C>0$
such that, for all $0\leq s\leq t\leq s+1$, it holds that
\[
\E[|X_{i}(t)-X_{i}(s)|^{\eta}]\leq C(t-s)^{\eta/\gamma_{i}}.
\]
\item[(b)] For each $\eta\in(0,1\wedge\delta_{i}]$ there exists a constant
$C>0$ such that
\[
\E[|X_{i}(t)-X_{i}^{\e}(t)|^{\eta}]\leq C\e^{\eta\kappa_{i}},\ \ t>0,\ \ \e\in(0,1\wedge t).
\]
\end{enumerate}
\end{Proposition} 
\begin{proof} \textit{(a)} This can be shown exactly
in the same way as Proposition \ref{PROP:00}.

\textit{(b)} Write $\E[|X_{i}(t)-X_{i}^{\e}(t)|^{\eta}]\leq R_{1}+R_{2}$,
where
\begin{align*}
R_{1} & =\E\left[\left|\int\limits _{t-\e}^{t}(b_{i}(X(u))-b_{i}(X(t-\e)))du\right|^{\eta}\right],\\
R_{2} & =\E\left[\left|\int\limits _{t-\e}^{t}(\sigma_{i}(X(t-))-\sigma_{i}(X(t-\e)))dZ_{i}(u)\right|^{\eta}\right].
\end{align*}
For the first term we use $|b_{i}(x)-b_{i}(y)|\leq C|x-y|^{\beta_{i}\wedge\delta}$
and part \textit{(a)} to obtain
\begin{align*}
R_{1} & \leq\E\left[\int\limits _{t-\e}^{t}|b_{i}(X(u))-b_{i}(X(t-\e))|du\right]^{\eta}\\
 & \leq C\e^{\eta}\sup\limits _{u\in[t-\e,t]}\E\left[|X(u)-X(t-\e)|^{\beta_{i}\wedge\delta}\right]^{\eta}\\
 & \leq C\e^{\eta}\sum\limits _{j=1}^{d}\sup\limits _{u\in[t-\e,t]}\E\left[|X_{j}(u)-X_{j}(t-\e)|^{\beta_{i}\wedge\delta}\right]^{\eta}\leq C\e^{\eta+\frac{\eta(\beta_{i}\wedge\delta)}{\gamma}}.
\end{align*}
For the second term we use Lemma \ref{LEMMA:00}.(a), $|\sigma_{i}(x)-\sigma_{i}(y)|\leq C|x-y|^{\chi_{i}\wedge(\delta/\gamma_{i})}$
and then part \textit{(a)} to obtain
\begin{align*}
R_{2} & \leq C\e^{\eta/\gamma_{i}}\sup\limits _{u\in[t-\e,t]}\E\left[|\sigma_{i}(X(u))-\sigma_{i}(X(t-\e))|^{\gamma_{i}}\right]^{\eta/\gamma_{i}}\\
 & \leq C\e^{\eta/\gamma_{i}}\sup\limits _{u\in[t-\e,t]}\E[|X(u)-X(t-\e)|^{\gamma_{i}(\chi_{i}\wedge(\delta/\gamma_{i}))}]^{\eta/\gamma_{i}}\\
 & \leq C\e^{\eta/\gamma_{i}}\sum\limits _{j=1}^{d}\sup\limits _{u\in[t-\e,t]}\E[|X_{j}(u)-X_{j}(t-\e)|^{\gamma_{i}(\chi_{i}\wedge(\delta/\gamma_{i}))}]^{\eta/\gamma_{i}}\leq C\e^{\frac{\eta}{\gamma_{i}}+\frac{\eta}{\gamma}\chi_{i}\wedge(\delta/\gamma_{i})}.
\end{align*}
This proves the assertion. \end{proof}

\subsection{Approximation when $\gamma_{i}\in(0,1)$}

In this Section we proceed similarly as in the proof of Proposition
\ref{PROP:04}. For $i\in\{1,\dots,d\}$ such that $\gamma_{i}\in(0,1)$,
define
\[
\kappa_{i}=\min\left\{ 1+\frac{\beta_{i}\wedge\chi_{i}}{\gamma},\frac{1}{\gamma_{i}}+\frac{\chi_{i}}{\max\{1,\gamma\}},\frac{1}{1-\rho}\right\} ,
\]
where $\rho=\min\{\beta_{j}\wedge\chi_{j}\ |\ j\in\{1,\dots,d\}\}$.
The following is our main estimate for this section. \begin{Proposition}\label{PROP:06}
Let $Z$ be a Lévy process with Lévy measure $\nu$ and symbol \eqref{SYMBOL}.
Suppose that \eqref{EQ:05} and (A3) are satisfied. Let $i\in\{1,\dots,d\}$
be such that $\gamma_{i}\in(0,1)$ and let $X$ be as in \eqref{EQ:12}.
Then
\begin{enumerate}
\item[(a)] For each $\eta\in(0,\delta_{i}]$ there exists a constant $C>0$
such that, for all $0\leq s\leq t\leq s+1$, it holds that
\[
\E[|X_{i}(t)-X_{i}(s)|^{\eta}]\leq C(t-s)^{\eta}.
\]
Moreover, if $\eta\in(0,1)$, then
\[
\E[|X_{i}(t)-X_{i}(s)|^{\eta}\wedge1]\leq C(t-s)^{\eta}.
\]
\item[(b)] For each $t>0$ and $\e\in(0,1\wedge t)$ there exists a $\F_{t-\e}$-measurable
random variable $U_{i}^{\e}(t)$ such that, setting
\[
X_{i}^{\e}(t)=U_{i}^{\e}(t)+\sigma_{i}(X(t-\e))(Z_{i}(t)-Z_{i}(t-\e)),
\]
for any $\eta\in(0,\delta_{i}]$ there exists a constant $C>0$ with
\[
\E\left[|X_{i}(t)-X_{i}^{\e}(t)|^{\eta}\right]\leq C\e^{\eta\kappa_{i}},\ \ t>0,\ \ \e\in(0,1\wedge t).
\]
\end{enumerate}
\end{Proposition} \begin{proof} \textit{(a)} This can be shown exactly
in the same way as Proposition \ref{PROP:04}.(a).

\textit{(b)} Define $Y_{i}(t)=Z_{i}(t)-t\int_{\R^{d}}\1_{\{|z|\leq1\}}z_{i}\nu(dz)$
and
\[
\widetilde{b}_{j}(x)=\begin{cases}
b_{j}(x), & \gamma_{j}\in[1,2]\\
b_{j}(x)-\sigma_{j}(x)\int\limits _{\R^{d}}\1_{\{|z|\leq1\}}z_{j}\nu(dz), & \gamma_{j}\in(0,1)
\end{cases},\ \ j\in\{1,\dots,d\}.
\]
Set $\tau=\e^{1/(1-\rho)}$ and define, for $s\in[t-\e,t]$, $s_{\tau}=t-\e+\tau\lfloor(s-(t-\e))/\tau\rfloor$.
Let $W^{\e}$ be the solution to
\[
W^{\e}(u)=X(t-\e)+\int\limits _{t-\e}^{u}\widetilde{b}(W^{\e}(s_{\tau}))ds,\ \ u\in[t-\e,t].
\]
Then $W^{\e}(t)$ is well-defined and $\F_{t-\e}$-measurable, because
it is a deterministic continuous function of $X(t-\e)$. Define
\begin{align}
X_{i}^{\e}(t) & =U_{i}^{\e}(t)+\sigma_{i}(X(t-\e))(Z_{i}(t)-Z_{i}(t-\e)),\label{APPROXIMATION:01}\\
U_{i}^{\e}(t) & =W_{i}^{\e}(t)+\e\sigma_{i}(X(t-\e))\int_{\R^{d}}\1_{\{|z|\leq1\}}z_{i}\nu(dz).\nonumber
\end{align}
It remains to prove the desired estimate. Thus write
\begin{align*}
X_{i}^{\e}(t)=X_{i}(t-\e)+\int\limits _{t-\e}^{t}\widetilde{b}_{i}(W^{\e}(s))ds+\int\limits _{t-\e}^{t}\sigma_{i}(X(t-\e))dY_{i}(s)+\int\limits _{t-\e}^{t}\left(\widetilde{b}_{i}(W^{\e}(s_{\tau}))-\widetilde{b}_{i}(W^{\e}(s))\right)ds,
\end{align*}
so that $\E[|X_{i}(t)-X_{i}^{\e}(t)|^{\eta}]\leq\E[I_{\e}^{\eta}]+\E[J_{\e}^{\eta}]+\E[J_{\e}^{\eta}]$
with
\begin{align*}
I_{\e} & =\int\limits _{t-\e}^{t}|\widetilde{b}_{i}(X(s))-\widetilde{b}_{i}(W^{\e}(s))|ds,\\
J_{\e} & =\left|\int\limits _{t-\e}^{t}(\sigma_{i}(X(s-))-\sigma_{i}(X(t-\e)))dY_{i}(s)\right|,\\
K_{\e} & =\int\limits _{t-\e}^{t}|\widetilde{b}_{i}(W^{\e}(s_{\tau}))-\widetilde{b}_{i}(W^{\e}(s))|ds
\end{align*}
For the second term we apply Lemma \ref{LEMMA:00}.(b), the Hölder
continuity of $\sigma$ and part \textit{(a)}, to obtain
\begin{align*}
\E[J_{\e}^{\eta}] & \leq C\e^{\eta/\gamma_{i}}\sup\limits _{s\in[t-\e,t]}\E[|\sigma_{i}(X(s))-\sigma_{i}(X(t-\e))|^{\gamma_{i}}]^{\eta/\gamma_{i}}\\
 & \leq C\e^{\eta/\gamma_{i}}\sup\limits _{s\in[t-\e,t]}\E[|X(s)-X(t-\e)|^{\gamma_{i}\chi_{i}}\wedge1]^{\eta/\gamma_{i}}\\
 & \leq C\e^{\eta/\gamma_{i}}\sum\limits _{j=1}^{d}\sup\limits _{s\in[t-\e,t]}\E[|X_{j}(s)-X_{j}(t-\e)|^{\gamma_{i}\chi_{i}}\wedge1]^{\eta/\gamma_{i}}\\
 & \leq C\sum\limits _{j=1}^{d}\left(\1_{[1,2]}(\gamma_{j})\e^{\frac{\eta}{\gamma_{i}}+\frac{\eta\chi_{i}}{\gamma_{j}}}+\1_{(0,1)}(\gamma_{j})\e^{\frac{\eta}{\gamma_{i}}+\chi_{i}}\right)\leq C\e^{\eta\kappa_{i}}.
\end{align*}
For the last term we use the $(\beta_{i}\wedge\chi_{i})$-Hölder continuity
of $\widetilde{b}_{i}$ to obtain
\begin{align*}
\E[K_{\e}^{\eta}]\leq C\E\left[\int\limits _{t-\e}^{t}|W^{\e}(s_{\tau})-W^{\e}(s)|^{\beta_{i}\wedge\chi_{i}}ds\right]^{\eta}\leq C\e\tau^{\eta(\beta_{i}\wedge\chi_{i})}=C\e^{\frac{\eta}{1-\rho}},
\end{align*}
where we have used $|W^{\e}(s)-W^{\e}(s_{\tau})|\leq C\tau$ (since
$\widetilde{b}$ is bounded), and $\beta_{i}\wedge\chi_{i}\geq\rho$.
Hence it remains to show that
\[
\E[I_{\e}^{\eta}]\leq C\left(\e^{\eta\left(1+\frac{\beta_{i}\wedge\chi_{i}}{\gamma}\right)}+\e^{\frac{\eta}{1-\rho}}\right).
\]
For this purpose, write
\[
W^{\e}(u)=X(t-\e)+\int\limits _{t-\e}^{u}\widetilde{b}(W^{\e}(s))ds+\int\limits _{t-\e}^{u}\left(\widetilde{b}(W^{\e}(s_{\tau}))-\widetilde{b}(W^{\e}(s))\right)ds,
\]
and let, for each $j\in\{1,\dots,d\}$,
\[
R_{j}^{\e}(t)=\begin{cases}
\left|\int\limits _{t-\e}^{t}\sigma_{j}(X(s-))dY_{j}(s)\right|, & \gamma_{j}\in(0,1)\\
\left|\int\limits _{t-\e}^{t}\sigma_{j}(X(s-))dZ_{j}(s)\right|, & \gamma_{j}\in[1,2]
\end{cases}.
\]
Then, for each $j\in\{1,\dots,d\}$, we obtain
\[
|X_{j}(t)-W_{j}^{\e}(t)|\leq\int\limits _{t-\e}^{t}|\widetilde{b}_{j}(W^{\e}(s_{\tau}))-\widetilde{b}_{j}(W^{\e}(s))|ds+\int\limits _{t-\e}^{t}|\widetilde{b}_{j}(X(s))-\widetilde{b}_{j}(W^{\e}(s))|ds+R_{j}^{\e}(t).
\]
Setting $S_{j}^{\e}(t)=\sup_{s\in[t-\e,t]}|X_{j}(s)-W_{j}^{\e}(s)|$,
using the $(\beta_{j}\wedge\chi_{j})$-Hölder continuity of $\widetilde{b}_{j}$
and that $|W^{\e}(s)-W^{\e}(s_{\tau})|\leq C\tau$, we obtain
\begin{align*}
S_{j}^{\e}(t) & \leq R_{j}^{\e}(t)+C\e\tau^{\beta_{j}\wedge\chi_{j}}+\sum\limits _{k=1}^{d}(C\e)S_{k}^{\e}(t)^{\beta_{j}\wedge\chi_{j}}\\
 & \leq R_{j}^{\e}(t)+C\e^{\frac{1}{1-\rho}}+\beta_{j}\wedge\chi_{j}\sum\limits _{k=1}^{d}S_{k}^{\e}(t)^{\beta_{j}\wedge\chi_{j}},
\end{align*}
where we have used $\beta_{j}\wedge\chi_{j}\geq\rho$, so that $\e\tau^{\beta_{j}\wedge\chi_{j}}\leq\e^{\frac{1}{1-\rho}}$,
and the Young inequality $xy\leq(1-\beta_{j}\wedge\chi_{j})x^{1/(1-\beta_{j}\wedge\chi_{j})}+\beta_{j}\wedge\chi_{j}y^{1/\beta_{j}\wedge\chi_{j}}$
with $x=C\e$ and $y=S_{k}^{\e}(t)^{\beta_{j}\wedge\chi_{j}}$. Taking
the sum over $j\in\{1,\dots,d\}$ and using that $\beta_{j}\wedge\chi_{j}<1$,
we obtain
\begin{align*}
\sum\limits _{j=1}^{d}S_{j}^{\e}(t)\leq C\sum\limits _{j=1}^{d}R_{j}^{\e}(t)+C\e^{\frac{1}{1-\rho}}.
\end{align*}
Using this inequality we obtain
\begin{align*}
I_{\e}\leq C\int_{t-\e}^{t}|X(s)-W^{\e}(s)|^{\beta_{i}\wedge\chi_{i}}ds\leq C\e\left(\sum\limits _{j=1}^{d}S_{j}^{\e}(t)\right)^{\beta_{i}\wedge\chi_{i}}\leq C\e\sum\limits _{j=1}^{d}\left(R_{j}^{\e}(t)^{\beta_{i}\wedge\chi_{i}}+\e^{\frac{\beta_{i}\wedge\chi_{i}}{1-\rho}}\right)
\end{align*}
Using now Lemma \ref{LEMMA:00} we obtain
\[
\E[I_{\e}^{\eta}]\leq C\e^{\eta}\left(\e^{\eta\frac{\beta_{i}\wedge\chi_{i}}{1-\rho}}+\E[R_{j}^{\e}(t)^{\eta(\beta_{i}\wedge\chi_{i})}]\right)\leq C\left(\e^{\frac{\eta}{1-\rho}}+\e^{\eta+\eta\frac{\beta_{i}\wedge\chi_{i}}{\gamma}}\right).
\]
This proves the assertion. \end{proof}

\subsection{Concluding the proof of Theorem \ref{THEOREM:03}}

As in the proof of Theorem \ref{THEOREM:01}, we start with a similar
esimate to Proposition \ref{PROP:01}. \begin{Proposition} Let $Z$
be a Lévy process with Lévy measure $\nu$ and symbol \eqref{SYMBOL}.
Suppose that \eqref{EQ:05}, (A1) and (A3) are satisfied. Let $X$
be as in \eqref{EQ:12}, take an anisotropy $a=(a_{i})_{i\in\{1,\dots,d\}}$
and $\eta\in(0,1)$ with
\begin{align}
\frac{\eta}{a_{j}}\leq1\wedge\delta,\ \ j\in\{1,\dots,d\}.\label{EQ:00}
\end{align}
Then there exists a constant $C=C(\eta)>0$ and $\e_{0}\in(0,1\wedge t)$
such that, for any $\e\in(0,\e_{0})$, $h\in[-1,1]$, $\phi\in C_{b}^{\eta,a}(\R^{d})$
and $i\in\{1,\dots,d\}$,
\[
\left|\E\left[|\sigma^{-1}(X(t))|^{-1}\Delta_{he_{i}}\phi(X(t))\right]\right|\leq C\|\phi\|_{C_{b}^{\eta,a}}\left(|h|^{\eta/a_{i}}\e^{\frac{\chi\wedge\delta}{\max\{1,\gamma\}}}+|h|\e^{-1/\alpha_{i}}+\max\limits _{j\in\{1,\dots,d\}}\e^{\eta\kappa_{j}/a_{j}}\right),
\]
where $\chi=\min\{\chi_{1},\dots,\chi_{d}\}$. \end{Proposition}
This proposition can be shown in the same way as Proposition \ref{PROP:01}.
A proof is therefore omitted. The proof of Theorem \ref{THEOREM:03}
is completed, provided we can show that Lemma \ref{LEMMA:02} is applicable
to the finite measure
\[
\mu_{t}(A)=\E[|\sigma^{-1}(X(t))|^{-1}\1_{A}(X(t))],\ \ t>0.
\]
This can be shown in exactly the same way as in the proof of Theorem
\ref{THEOREM:01}.

\section{Smoothing property of the Lévy noise}

The following is a simple modification of \citep[Lemma 3.3]{DF13},
it provides a general estimate on the derivative of the density
of a Lévy process. 
\begin{Proposition}\label{PROP:03} 
Let $Z$ be
a Lévy process with Lévy measure $\nu$ and symbol \eqref{SYMBOL}.
Suppose that
\begin{align}
\liminf\limits _{|\xi|\longrightarrow\infty}\frac{\mathrm{Re}(\Psi_{\nu}(\xi))}{\log(1+|\xi|)}=\infty,\label{EQ:04}
\end{align}
and assume that there exists $t_{0}>0$ and $C_{1}>0$ such that
\begin{align}
\int\limits _{\R^{d}}e^{-t\mathrm{Re}(\Psi_{\nu}(\xi))}|\xi|^{d+2}d\xi\leq C_{1}\Xi(t)^{2d+2},\ \ t\in(0,t_{0}),\label{EQ:29}
\end{align}
where $\delta(\eta)=\sup_{|\xi|\leq\eta}\mathrm{Re}(\Psi_{\nu}(\xi))$
and $\Xi(t)=\delta^{-1}(1/t)$. Then $Z(t)$ has a density $g_{t}\in C^{1}(\R^{d})$
such that $\nabla g_{t}\in L^{1}(\R^{d})$ and for some constant $C_{2}>0$ and $t_1 > 0$,
\begin{align}
\|\nabla g_{t}\|_{L^{1}(\R^{d})}\leq C_{2}\Xi(t\wedge t_1),\ \ t > 0.\label{APPROX:04}
\end{align}
\end{Proposition} 
\begin{proof}
From \eqref{EQ:04} it follows that $Z(t)$ has a density $p_t$. For $r > 0$
 write $\Psi_{\nu} = \Psi_{\nu_r} + \Psi_{\nu_r'}$ where $\nu'_r(dz) = \1_{ \{ |z| > r\} }\nu(dz)$ and $\nu_r(dz) = \1_{ \{ |z| \leq r \} }\nu(dz)$.
 Then $p_t = q_t^r \ast p_t^r$, where $q_t^r$ is the infinite divisible distribution with symbol $\Psi_{\nu_r'}$ and $p_t^r$ the density with symbol $\Psi_{\nu_r}$. 
 It follows from \cite[Proposition 2.3]{SSW12} that there exist $t_1 > 0$ and $C > 0$ such that for all $t \in (0,t_1]$,
 \[
  | \nabla p_t^{1/\Xi(t)}(z)| \leq C \Xi(t)^{d+1}\left(1 + |z|\Xi(t)\right)^{-d-1}, \qquad z \in \R^d,
 \]
 and hence $\| \nabla p_t^{1 / \Xi(t)}\|_{L^1(\R^d)} \leq C \Xi(t) \int_{\R^d} (1 + |z|)^{-d - 1} dz < \infty$.
 By Young inequality we obtain for $t \in (0,t_1]$ and some constant $C' > 0$,
 \[
  \| \nabla p_t \|_{L^1(\R^d)} = \| q_t^{1 / \Xi(t)} \ast (\nabla p_t^{1/\Xi(t)}) \|_{L^1(\R^d)} \leq \| \nabla p_t^{1/\Xi(t)} \|_{L^1(\R^d)} \leq C' \Xi(t).
 \]
 Now let $t > t_1$, then using the infinite divisibility of $p_t$, we obtain $p_t = p_{t-t_1} \ast p_{t_1}$ and hence
 \[
  \| \nabla p_t \|_{L^1(\R^d)} = \| p_{t-t_1} \ast (\nabla p_{t_1}) \|_{L^1(\R^d)} \leq \| \nabla p_{t_1} \|_{L^1(\R^d)} \leq C' \Xi(t_1).
 \]
\end{proof}

Below we provide a sufficient condition for (A1)
which includes Example \ref{MAIN:EXAMPLE}. \begin{Proposition} Let
$m\leq d$ and $I_{1},\dots,I_{m}\subset\{1,\dots,d\}$ be disjoint
with $I_{1}\cup\dots\cup I_{d}=\{1,\dots,d\}$. For each $j\in\{1,\dots,m\}$,
let $Z^{j}$ be a $|I_{j}|$-dimensional Lévy process with Lévy measure
$\nu_{j}$ and symbol
\[
\Psi_{\nu_{j}}(\xi)=\int\limits _{\R^{|I_{j}|}}\left(1+i\1_{\{|z|\leq1\}}(\xi\cdot z)-e^{i(\xi\cdot z)}\right)\nu_{j}(dz).
\]
Suppose that $Z^{1},\dots,Z^{m}$ are independent and there exist
constants $c,C>0$ and $\alpha_{1},\dots,\alpha_{m}\in(0,2)$ such
that
\[
c|\xi|^{\alpha_{j}}\leq\mathrm{Re}(\Psi_{\nu_{j}}(\xi))\leq C|\xi|^{\alpha_{j}},\ \ |\xi|\gg1,\ \ j\in\{1,\dots,m\}.
\]
Then $Z=(Z^{1},\dots,Z^{m})$ has a smooth density $f_{t}$ such that
\[
\int\limits _{\R^{d}}|\nabla_{I_{j}}f_{t}(z)|dz\leq Ct^{-1/\alpha_{j}},\ \ t\to0,\ \ j\in\{1,\dots,m\}.
\]
In particular, $Z$ satisfies condition (A1). 
\end{Proposition}

The following are our main examples for condition (A1). 
\begin{Example}\label{EXAMPLE:07}
Let $Z_{1},\dots,Z_{d}$ be independent one-dimensional pure-jump
Lévy processes with Lévy measures $\nu_{1},\dots,\nu_{d}$.
\begin{enumerate}
\item[(i)] Let $c_{1}^{\pm},\dots,c_{d}^{\pm}\geq0$, $\alpha_{1}^{\pm},\dots,\alpha_{d}^{\pm}\in(0,2)$,
and assume that
\[
\nu_{k}=c_{k}^{+}r^{-1-\alpha_{k}^{+}}\1_{(0,1]}(r)dr+c_{k}^{-}|r|^{-1-\alpha_{k}^{-}}\1_{[-1,0)}(r)dr+\mu_{k},
\]
where $\mu_{1},\dots,\mu_{d}$ are one-dimensional Lévy measures.
If $c_{k}^{+}+c_{k}^{-}>0$ holds for each $k=1,\dots,d$,
then previous proposition is applicable and hence (A1) holds for
\begin{align*}
\alpha_{k}=\begin{cases}
\alpha_{k}^{+}, & \text{ if }c_{k}^{+}\neq0,\ c_{k}^{-}=0,\\
\alpha_{k}^{-}, & \text{ if }c_{k}^{+}=0,\ c_{k}^{+}\neq0\\
\max\{\alpha_{k}^{+},\alpha_{k}^{-}\}, & \text{ if }c_{k}^{+}\neq0,\ c_{k}^{-}\neq0.
\end{cases},
\end{align*}
\item[(ii)] Let $\nu_{1},\dots,\nu_{d}$ be, for $\alpha_{1},\dots,\alpha_{d}\in(0,2)$,
given by
\[
\nu_{k}(dz)= \sum \limits_{n=1}^{\infty} n^{\alpha_{k}-1}\delta_{1/n}(dz),\qquad k=1,\dots,d.
\]
Then (A1) is satisfied (see \citep[Example 1.6]{DF13}).
\end{enumerate}
\end{Example} 
Condition (A1) holds also true for the case where each $Z_{k}$ is a subordinate Brownian motion. 
\begin{Example}
Let $S_{1}(t),\dots,S_{d}(t)$ be independent subordinators with Laplace
exponents
\[
\psi_{k}(\lambda)=\lambda^{\alpha_{k}/2}\left(\log(1+\lambda)\right)^{\beta_{k}/2},\ \ \alpha_{k}\in(0,2),\ \ \beta_{k}\in(-\alpha_{k},2-\alpha_{k}),\ \ k=1,\dots,d.
\]
Let $B(t)$ a $d$-dimensional Brownian motion independent of $S_{1},\dots,S_{d}$,
and define $Z_{k}(t)=B_{k}(S_{k}(t))$, $k=1,\dots,d$. Let $\Psi_{k}$,
$k=1,\dots,d$, be the corresponding symbolds. 
Then $Z = (Z_1,\dots, Z_d)$ satisfies \eqref{EXAMPLE:30} with
\[
 \widetilde{\alpha}_k = \begin{cases} \alpha_k, & \beta_k \in [0, 2 - \alpha_k) \\ 
 \alpha_k - \e_k, & \beta_k \in (-\alpha_k, 0) \end{cases}, \ \ k \in \{1,\dots, d\},
\]
for any choice of $\e_k \in (0, \alpha_k)$.
\end{Example}
\begin{proof}\
 It suffices to show that Proposition \ref{PROP:03} is applicable in $d = 1$ to each $Z_k$ with symbol $\Psi_k$, $k \in \{1,\dots, d\}$.
 Fix $k \in \{1,\dots, d\}$ and observe that, by \cite[Example 1.4]{SSW12}, one has
\[
\mathrm{Re}(\Psi_{k}(\xi)) \simeq |\xi|^{\alpha_{k}}\left(\log(1+|\xi|)\right)^{\beta_{k}/2}, \qquad \text{ as } |\xi| \to \infty.
\]
Hence we easily see that \eqref{EQ:04} is satisfied.
 It follows from the proof of \cite[Example 1.4]{SSW12} and \cite[Theorem 1.3]{SSW12} that
\eqref{EQ:29} is satisfied and, moreover, one has
\[
  \Xi_k(t) \simeq t^{- \frac{1}{\alpha_k} } \left( \log\left( 1 + \frac{1}{t}\right) \right)^{- \frac{\beta_k}{2 \alpha_k}}, \qquad \text{ as } |\xi| \to \infty.
 \]
 If $\beta_k \in [0,2-\alpha_k)$, then we may use $\log\left( 1 + t^{-1}\right) \geq \log(2)$ for $t \in (0,1]$, to obtain
 \[
  \Xi_k(t) \leq C t^{- \frac{1}{\alpha_k}} \left(\log(2)\right)^{- \frac{\beta_k}{2\alpha_k}}, \qquad \text{ as } t \to 0.
 \]
  If $\beta_k \in (-\alpha_k,0)$, then we may use $\log\left( 1 + t^{-1} \right) \leq Ct^{ - r_k}$,
 where $r_k = \frac{2 \alpha_k}{|\beta_k|} \frac{\e_k}{\alpha_k (\alpha_k - \e_k)}$, to obtain
 \[
  \Xi_k(t) \leq C t^{- \frac{1}{\alpha_k}} \left( \log\left( 1 + \frac{1}{t}\right) \right)^{\frac{|\beta_k|}{2\alpha_k}}
  \leq C t^{- \frac{1}{\widetilde{\alpha}_k}}, \qquad \text{ as } t \to 0.
 \]
\end{proof} 

\section{Appendix}

Below we state some useful estimates on stochastic integrals with
respect to Lévy processes. Similar results were obtained in \citep[Lemma 5.2]{DF13}.
\begin{Lemma}\label{LEMMA:00} Let $Z$ be a Lévy process with Lévy
measure $\nu$ and symbol given by \eqref{SYMBOL}. Suppose that there
exist $\gamma\in(0,2]$ and $\delta\in(0,\gamma]$ such that
\[
\int\limits _{\R^{d}}\left(\1_{\{|z|\leq1\}}|z|^{\gamma}+\1_{\{|z|>1\}}|z|^{\delta}\right)\nu(dz)<\infty.
\]
Then the following assertions hold.
\begin{enumerate}
\item[(a)] Let $0<\eta\leq\delta\leq\gamma$ and $1\leq\gamma\leq2$. Then there
exists a constant $C>0$ such that, for any predictable process $H(u)$
and $0\leq s\leq t\leq s+1$,
\[
\E\left[\left|\int\limits _{s}^{t}H(u)dZ(u)\right|^{\eta}\right]\leq C(t-s)^{\eta/\gamma}\sup\limits _{u\in[s,t]}\E\left[|H(u)|^{\gamma}\right]^{\eta/\gamma}.
\]
\item[(b)] Let $0<\eta\leq\delta\leq\gamma<1$ and define $Y(t):=Z(t)+t\int_{|z|\leq1}z\nu(dz)$.
Then there exists a constant $C>0$ such that, for any predictable
process $H(u)$ and $0\leq s\leq t\leq s+1$,
\[
\E\left[\left|\int\limits _{s}^{t}H(u)dY(u)\right|^{\eta}\right]\leq C(t-s)^{\eta/\gamma}\sup\limits _{u\in[s,t]}\E[|H(u)|^{\gamma}]^{\eta/\gamma}.
\]
Moreover, setting $\Delta Y(u)=\lim_{r\nearrow u}(Y(u)-Y(r))$, we
obtain
\[
\E\left[\left(\sum\limits _{u\in[s,t]}|\Delta Y(u)|\right)^{\eta}\right]\leq C(t-s)^{\eta/\gamma}.
\]
\end{enumerate}
\end{Lemma} \begin{proof} (a) Let $N(du,dz)$ be a Poisson random
measure with compensator $\widehat{N}(du,dz)=dum(dz)$ such that
\[
Z(t)=\int\limits _{0}^{t}\int\limits _{|z|\leq1}z\widetilde{N}(du,dz)+\int\limits _{0}^{t}\int\limits _{|z|>1}zN(du,dz),
\]
where $\widetilde{N}(du,dz)=N(du,dz)-\widehat{N}(du,dz)$ denotes
the corresponding compensated Poisson random measure. Then
\begin{align*}
\E\left[\left|\int\limits _{s}^{t}H(u)dZ(u)\right|^{\eta}\right] & \leq C\E\left[\left|\int\limits _{s}^{t}\int\limits _{|z|\leq1}H(u)z\widetilde{N}(du,dz)\right|^{\eta}\right]+C\E\left[\left|\int\limits _{s}^{t}\int\limits _{|z|>1}H(u)zN(du,dz)\right|^{\eta}\right].
\end{align*}
If $\eta\geq1$, then by the BDG-inequality, sub-additivity of $x\longmapsto x^{\frac{\gamma}{2}}$
and Hölder inequality we obtain
\begin{align*}
 & \ \E\left[\left|\int\limits _{s}^{t}\int\limits _{|z|\leq1}H(u)z\widetilde{N}(du,dz)\right|^{\eta}\right]\leq C\E\left[\left|\int\limits _{s}^{t}\int\limits _{|z|\leq1}|H(u)z|^{2}N(du,dz)\right|^{\eta/2}\right]\\
 & \leq C\E\left[\left|\int\limits _{s}^{t}\int\limits _{|z|\leq1}|H(u)z|^{\gamma}N(du,dz)\right|^{\eta/\gamma}\right]\leq C\E\left[\int\limits _{s}^{t}\int\limits _{|z|\leq1}|H(u)z|^{\gamma}du\nu(dz)\right]^{\eta/\gamma}\\
 & \leq C(t-s)^{\frac{\eta}{\gamma}}\left(\int\limits _{|z|\leq1}|z|^{\gamma}\nu(dz)\right)^{\eta/\gamma}\sup\limits _{u\in[s,t]}\E\left[|H(u)|^{\gamma}\right]^{\eta/\gamma}.
\end{align*}
If $0<\eta\leq1\leq\gamma\leq2$, then the Hölder inequality and previous
estimates imply
\begin{align*}
\E\left[\left|\int\limits _{s}^{t}\int\limits _{|z|\leq1}H(u)z\widetilde{N}(du,dz)\right|^{\eta}\right] & \leq\E\left[\left|\int\limits _{s}^{t}\int\limits _{|z|\leq1}H(u)z\widetilde{N}(du,dz)\right|\right]^{\eta}\\
 & \leq C(t-s)^{\frac{\eta}{\gamma}}\sup\limits _{u\in[s,t]}\E\left[|H(u)|^{\gamma}\right]^{\eta/\gamma}.
\end{align*}
Let us turn to the integral involving the big jumps. If $\delta\in(0,1]$,
then by sub-additivity of $x\longmapsto x^{\delta}$ and Hölder inequality
we get
\begin{align*}
 & \ \E\left[\left|\int\limits _{s}^{t}\int\limits _{|z|>1}H(u)zN(du,dz)\right|^{\eta}\right]\leq\E\left[\left|\int\limits _{s}^{t}\int\limits _{|z|>1}|H(u)z|^{\delta}N(du,dz)\right|^{\eta/\delta}\right]\\
 & \leq\E\left[\int\limits _{s}^{t}\int\limits _{|z|>1}|H(u)z|^{\delta}du\nu(dz)\right]^{\eta/\delta}\leq(t-s)^{\frac{\eta}{\delta}}\sup\limits _{u\in[s,t]}\E\left[|H(u)|^{\delta}\right]^{\eta/\delta}.
\end{align*}
If $\delta\in(1,\gamma]$, then
\begin{align*}
 & \ \E\left[\left|\int\limits _{s}^{t}\int\limits _{|z|>1}H(u)zN(du,dz)\right|^{\eta}\right]\leq C\E\left[\left|\int\limits _{s}^{t}\int\limits _{|z|>1}H(u)z\widetilde{N}(du,dz)\right|^{\eta}\right]+C\E\left[\left|\int\limits _{s}^{t}\int\limits _{|z|>1}H(u)zdu\nu(dz)\right|^{\eta}\right]
\end{align*}
The stochastic integral can be estimated similarly as before, which
gives
\[
\E\left[\left|\int\limits _{s}^{t}\int\limits _{|z|>1}H(u)z\widetilde{N}(du,dz)\right|^{\eta}\right]\leq C(t-s)^{\eta/\delta}\sup\limits _{u\in[s,t]}\E[|H(u)|^{\delta}]^{\eta/\delta}.
\]
The second integral can be estimated by
\begin{align*}
\E\left[\left|\int\limits _{s}^{t}\int\limits _{|z|>1}H(u)zdu\nu(dz)\right|^{\eta}\right]\leq\E\left[\left|\int\limits _{s}^{t}\int\limits _{|z|>1}H(u)zdu\nu(dz)\right|^{\delta}\right]^{\eta/\delta}\leq(t-s)^{\eta}\sup\limits _{u\in[s,t]}\E\left[|H(u)|^{\delta}\right]^{\eta/\delta}.
\end{align*}
Collecting all estimates and using $t-s\leq1$, so that $(t-s)^{\eta/\delta}\leq(t-s)^{\eta/\gamma}$
and $(t-s)^{\eta}\leq(t-s)^{\eta/\gamma}$, gives
\[
\E\left[\left|\int\limits _{s}^{t}H(u)dZ(u)\right|^{\eta}\right]\leq C(t-s)^{\frac{\eta}{\gamma}}\left(\sup\limits _{u\in[s,t]}\E\left[|H(u)|^{\gamma}\right]^{\eta/\gamma}+\sup\limits _{u\in[s,t]}\E[|H(u)|^{\delta}]^{\eta/\delta}\right).
\]
Applying Hölder inequality with $p=\gamma/\delta$ and $q=\gamma/(\gamma-\delta)$
gives the assertion.

(b) Consider the decomposition
\[
Z(t)=\int\limits _{0}^{t}\int\limits _{|z|\leq1}zN(du,dz)+\int\limits _{0}^{t}\int\limits _{|z|>1}zN(du,dz)-(t-s)A,
\]
where $A=\int_{|z|\leq1}z\nu(dz)$. Then we obtain
\begin{align*}
\E\left[\left|\int\limits _{s}^{t}H(u)dZ(u)\right|^{\eta}\right] & \leq\E\left[\left|\int\limits _{s}^{t}\int\limits _{|z|\leq1}H(u)zN(du,dz)\right|^{\eta}\right]+\E\left[\left|\int\limits _{s}^{t}\int\limits _{|z|>1}H(u)zN(du,dz)\right|^{\eta}\right]\\
 & \ \ \ +\E\left[\left|\int\limits _{s}^{t}H(u)Adu\right|^{\eta}\right]\\
 & \leq C(t-s)^{\eta/\gamma}\sup\limits _{u\in[s,t]}\E[|H(u)|^{\gamma}]^{\eta/\gamma}+C(t-s)^{\eta/\delta}\sup\limits _{u\in[s,t]}\E[|H(u)|^{\delta}]^{\eta/\delta}\\
 & \ \ \ +C(t-s)^{\eta}\sup\limits _{u\in[s,t]}\E[|H(u)|^{\eta}]\\
 & \leq C(t-s)^{\eta/\gamma}\sup\limits _{u\in[s,t]}\E[|H(u)|^{\gamma}]^{\eta/\gamma},
\end{align*}
where we have used the same estimates as in part (a). \end{proof}

\begin{footnotesize}

\bibliographystyle{alpha}
\bibliography{Bibliography}

\end{footnotesize}

\end{document}